\documentclass[12pt]{amsart}
\usepackage{latexsym}
\usepackage{amsfonts}
\usepackage{amsmath}
\usepackage{amsthm}
\usepackage{amssymb}
\usepackage{amscd}

\usepackage[english]{babel}
\newcommand{\MnC}{M_n(\mathbb{C})}
\newcommand{\GL}{\mathrm{GL}}

\newcommand{\CC}{\mathbb{C}}
\newcommand{\RR}{\mathbb{R}}
\newcommand{\ZZ}{\mathbb{Z}}

\newcommand{\gotG}{\mathfrak{g}}

\newcommand{\cC}{\mathcal{C}}
\newcommand{\cH}{\mathcal{H}}
\newcommand{\cS}{\mathcal{S}}
\newcommand{\cW}{\mathcal{W}}
\newcommand{\Hom}{\mathrm{Hom}}
\newtheorem{thm}{Theorem}[section]
\newtheorem{thm*}{Theorem}
\newtheorem{lem}[thm]{Lemma}
\newtheorem{lem*}[thm*]{Lemma}

\newtheorem{cor}[thm]{Corollary}
\newtheorem{prop*}[thm*]{Proposition}
\newtheorem{cor*}[thm*]{Corollary}

\newtheorem{defn*}[thm*]{Definition}
\newtheorem{thm-defn}[thm]{Theorem-Definition}

\theoremstyle{remark}
\newtheorem{rem}[thm]{Remark}

\newcommand{\beq}{\begin{equation}}
\newcommand{\enq}{\end{equation}}
\newcommand{\beqn}{\begin{equation*}}
\newcommand{\enqn}{\end{equation*}}
\begin{document}
\author{Alexander Kemarsky}
\address{Mathematics Department, Technion - Israel Institute of Technology, Haifa, 32000 Israel}
\email{alexkem@tx.technion.ac.il}

\title{Gamma Factors of Distinguished Representations of $\GL_n(\CC)$}
\begin{abstract}
Let $(\pi,V)$ be a $\GL_n(\RR)$-distinguished, irreducible, admissible representation of $\GL_n(\CC)$,
let $\pi'$ be an irreducible, admissible, $\GL_m(\RR)$-distinguished
representation of $\GL_m(\CC)$,
and let $\psi$ be a non-trival character of $\CC$ which is
trivial on $\RR$.
We prove that the Rankin-Selberg gamma factor
at $s=\frac{1}{2}$ is $\gamma(\frac{1}{2},\pi \times \pi'; \psi) = 1$. The
result follows as a simple consequence from the characterisation of $\GL_n(\RR)$-distinguished representations in terms of their Langlands data.

\end{abstract} 

\date{\today}
\maketitle
\makeatletter
\@setabstract

\numberwithin{equation}{section}

\tableofcontents

\begin{section}{Introduction}
Let $G_n(\CC) = \GL_n(\CC)$, $G_n(\RR) = \GL_n(\RR)$. Let $B_n = B_n(\CC)$ be the Borel subgroup of upper
triangular matrices in $G_n(\CC)$. Denote complex conjugation by $x \to \bar{x}$. We identify $G_n(\CC)/G_n(\RR)$ with the space of matrices
$$X_n =  \bigg\{ x \in G_n(\CC)  \big| x \cdot \bar{x} = I_n \bigg\},$$
via the isomorphism $gG_n(\RR) \mapsto g \cdot \bar{g}^{-1} $, see \cite[Chapter 3, Section 1, Lemma 1]{Serre}
for the proof of the surjectivity of this map.
Given a representation $\pi$ of $G_n(\CC)$,
the representation $\bar{\pi}$ is defined by the formula $\bar{\pi}(g) := \pi(\bar{g})$.\\
The group $G_n(\CC)$ acts on $X_n$ by the twisted conjugation, where the action is induced
by the natural action $l(g)g'G_n(\RR) := gg'G_n(\RR)$. Namely, we have
$$g'G_n(\RR)  \leftrightarrow
g' \cdot \bar{g'}^{-1} := x \text{ and }
l(g)(g'G_n(\RR)) := gg'G_n(\RR) \leftrightarrow gg'\bar{g'}^{-1}\bar{g}^{-1}.$$ Hence, the action of $G_n(\CC)$ on $X$ is
given by $l(g) x := gx\bar{g}^{-1}$.\newline
For a topological vector space $V$, we denote by $V^*$ the topological dual of $V$, i.e.,
the space of all continuous maps from $V$ to $\CC$.
In this paper we work with the category of the admissible smooth Fr\'echet representations of moderate growth,
 see \cite[Section 11.5]{WallachB2}, see also \cite[Section 2.1]{AGS}. 
 \newline
A representation $(\pi,V)$ of $G_n(\CC)$ is called $G_n(\RR)$-distinguished if there exists
a non-zero continuous linear map
$L:V \rightarrow \CC $, such that $$L(\pi(h)v) = L(v) \hspace{3mm} \forall v \in V, h \in G_n(\RR).$$
We denote the space of all such linear maps by $\big(V^* \big)^{G_n(\RR)}.$ \newline
Let $\psi:\CC \to \CC^{\times}$ be a non-trivial unitary character which is trivial on $\RR$, for example
$$\psi(x) = e^{\pi (x-\bar{x})}.$$
We let $U_n(\CC)$ be the group of upper triangular matrices with unit diagonal
and we denote by $\theta_{\psi,n}$ the character $\theta_{\psi,n}:U_n(\CC) \to \CC^{\times}$ defined
by $$\theta_{\psi,n}(u) = \psi\left(\sum_{i=1}^{n-1} u_{i,i+1} \right).$$
A $\psi$-form on $V$ is a nonzero continuous linear form $\lambda:V\to \CC$ such that
$$\lambda(\pi(u)v) = \theta_{\psi,n}(u)v, $$
for each $v \in V$ and each $u \in U_n(\CC)$.
We say that $\pi$ is a generic representation if there exists a $\psi$-form on $V$.
We prove the following theorem
\begin{thm}\label{thm: gamma factor}
Let $\pi$ be an irreducible, $G_t(\RR)$-distinguished
representation of $G_t(\CC)$, and $\pi'$ be an irreducible, $G_r(\RR)$-distinguished
representation of $G_r(\CC)$. Then the value of the Rankin-Selberg gamma factor
at $s=\frac{1}{2}$ is
$$ \gamma\left(\frac{1}{2} ,\pi \times \pi'; \psi \right) = 1.$$
\end{thm}
See Section \ref{section: Rankin-Selberg} for the definition of Rankin-Selberg integrals and Rankin-Selberg
gamma factors.\\
We will deduce Theorem \ref{thm: gamma factor} from the characterization of
irreducible  $G_n(\RR)$-distinguished representations of $G_n(\CC)$.
Actually, we will prove the following analogue of \cite[theorem B.1]{AL}
\begin{thm}\label{thm:dist implies main}
Let $(\pi,V)$ be an irreducible, admissible, generic and $G_n(\RR)$-distinguished representation of $G_n(\CC)$.
Then $\bar{\pi} \simeq \tilde{\pi}$, where $\tilde{\pi}$ is the contragredient representation of $\pi$.
\end{thm}
Let $\chi$ be a character of $B_n$. We will denote by $I(\chi)$ the normalized
parabolic induction representation
$I(\chi) := Ind_{B_n}^{G_n(\CC)}(\chi)$ of the character $\chi = (\chi_1, ..., \chi_n)$ from $B_n(\CC)$ to  $G_n(\CC)$. We remind to the reader that this space consists of smooth functions such that
$f(bg) = \left(\chi \delta^{1/2}\right)(b)f(g)$ for all $b \in B_n(\CC)$ and all $g \in G_n(\CC)$. The group
$G_n(\CC)$ acts on $I(\chi)$ by right translations.
Let $S_n$ be the group of permutations on $n$ elements. It acts naturally on the characters of $B_n$.
Theorem \ref{thm:dist implies main} follows from the following
\begin{thm}\label{thm:induced representation}
Let $(\pi,V)$ be an irreducible and
$G_n(\RR)$-distinguished representation of $G_n(\CC)$.
Let $\chi = (\chi_1,\chi_2,...,\chi_n)$ be a character of $B_n$ such that
$$|\chi(t) | = |t_1|^{\lambda_1} |t_2|^{\lambda_2} ... |t_n|^{\lambda_n} \text{  with  }
\lambda_1 \geq \lambda_2 \geq ... \geq \lambda_n. $$
Suppose $\pi$ is the Langlands
quotient of $I(\chi)$, that is, the unique irreducible quotient
of $I(\chi)$. Then there exists an involution $w \in S_n$ such that $w\chi = \overline{(\chi^{-1})}$.
Moreover, we can choose this $w$ such that for every fixed point $i$ of $w$ we have $\chi_i(-1)=1$.
\end{thm}
\begin{rem} a) Note that the conditions $w(i)=i$,
$\chi_{w(i)}=\bar{\chi_i}^{-1}$, and $\chi_i(-1)=1$ imply
that $\chi_i$ is $\GL_1(\RR)$-distinguished. Indeed, $\chi_i = \bar{\chi_i}^{-1}$ implies that $\chi_i$ is $\RR_{+}$-invariant. Together with the condition $\chi_i(-1)=1$ this means that $\chi_i$ is $\RR^{\times}(=\GL_1(\RR) )$-invariant. \\
b) Let $(\pi,V)$ be an irreducible representation of $G_n(\CC)$. The existence of
$I(\chi)$ with the properties stated in Theorem \ref{thm:induced representation}
is a well-known fact, see \cite[Theorem 5.4.1]{WallachB1}. Also, for a
given $\pi$, such a character $\chi$ is unique. It is well-known that the Langlands
quotient of $I(\chi)$ is generic if and
only if $I(\chi)$ is irreducible, see Appendix. Therefore,
Theorem \ref{thm:dist implies main} follows from Theorem \ref{thm:induced representation}.
\end{rem}
A similar result was proven by Marie-Noelle Panichi in her Ph.D., see \cite[Theorem 3.3.6]{Pan}.
As an application of our classification we deduce Theorem \ref{thm: gamma factor}.
In section \ref{sec: functionals} we prove a new type of integral identity for Whittaker functions
on generic $G_n(\RR)$-distinguished representations. This proves \cite[Assumption 5.2]{LapMao}. In the $p$-adic case a
similar identity was proven in \cite[Corollary 7.2]{O}. Our proof is similar to the proof in the $p$-adic case,
but in the archimedean case there are many analytical difficulties. We overcome them
in sections \ref{section: Rankin-Selberg}-\ref{sec:Asai}. \\
Finally, in Appendix B we prove a converse type theorem. We prove that if $(\pi,V) = Ind(\chi)$ is
an irreducible, generic, admissible unitary representation of $G_n(\CC)$ such that for every unitary character 
$\chi'(z) = (z/|z|)^{2m}$ with $m \in \ZZ$ we have $$\gamma(\frac{1}{2}, \pi \times \chi', \psi) = 1,$$ then
$\pi$ is $G_n(\RR)$-distinguished. The proof is done by a combinatorics argument combined with the Tadic-Vogan classification of
the unitary dual of $G_n(\CC)$.
\subsection*{Acknowledgements}
I would like to thank Omer Offen for posing to me this question and providing many explanations
of the subject. \newline
I am grateful to Dmitry Gourevitch and Erez Lapid for many fruitful discussions and
their help. \newline
During the conference in Jussieu, June 2014 I told the results of this paper to Herve Jacquet.
I would like to thank him, his comments were very helpful. \\
I also wish to thank Avraham Aizenbud, Uri Bader, Max Gurevich, Job Kuit, Nadir Matringe, Amos Nevo, Henrik Schlictkrull
and Dror Speiser on useful conversations and remarks.
\newline The research was supported by ISF grant No.  1394/12.
\end{section} 
\begin{section}{Notation and preliminaries}
Let $F$ be either $\RR$ or $\CC$. Let $M(a \times b, F)$ be the space of matrices with
$a$ rows and $b$ columns with entries in $F$.
Let $\eta_n = (0,0, \hdots, 1) \in M(1 \times n, \RR)$.
Let $P_n(\RR)$ be the subgroup of $G_n(\RR)$ consisting of all $n \times n$ matrices
with the last row equal to $\eta_n$.\\
Let $U_n(F)$ be the group of all upper triangular matrices in $M(n \times n, F)$
with unit diagonal. Let $$K_n=\{ g \in G_n(\CC) : g \cdot ^t\hspace{-0.6mm} \bar{g} = I \}$$ be the standard maximal compact subgroup of $G_n(\CC)$.\\
For $V$ a finite dimensional vector space over $\RR$ we denote by $\cS(V)$ the Schwartz space of all
infinitely differentiable functions $f:V \to \CC$ of rapid decay. \\
Let $\Phi \in \cS(V)$, where $V = M(a \times b, \CC) $. We denote by $\hat{\Phi}$ the Fourier transform of $\Phi$.
It is a function on the same space, defined by
$$ \hat{\Phi}(X) = \int \Phi(Y) \psi(-Tr(^t \hspace{-1mm} X Y)) dY.$$
For $\Phi \in \cS(\CC^n)$ and $g \in G_n(\CC)$ we denote by $(R(g)\Phi)(x):=\Phi(xg)$ the right translation
of $\Phi$ by $g$.\\
For $z =x+iy \in \CC$ we denote by $|z|=\sqrt{x^2+y^2}$ the usual absolute value of $z$ and by
$|z|_{\CC}=|z|^2 = x^2 + y^2$ the square of the usual absolute value. Note that $\mu(zA) =|z|_{\CC} \mu(A)$, where $A \subset \CC$ is an open set and $\mu$ is a Haar measure on $\CC$. \\
Let $W_{n}=S_{n}$ and let $W_{n,2}=\{w\in W_{n}:w^{2}=1\}$ be the set of involutions in $W_{n}$.
For $w \in W_{n,2}$ set $$I_w =\{(i,j): i > j, w(i) > w(j)\},$$ and define for any
function $\kappa:I_{w} \to \ZZ_{\geq 0}$ a character $\alpha_{\kappa}$ of $B_n$ by the formula
$$\alpha_{\kappa}(\text{diag}(t_1,...,t_n)) = \prod_{(i,j) \in I_w} {\big[{\frac{t_i }{t_j}}\big]^{\kappa(i,j)}}. $$
We will identify $\alpha_{\kappa}$ with the one-dimensional representation of $B_n$
on the vector space $\CC$ with the action of $\alpha_{\kappa}$. By abuse of notation
we will denote both the function and the one-dimensional representation by the same
letter $\alpha_{\kappa}$. \\
We refer the reader to the notation of \cite{AL}. For the convenience of the reader we write here
notation and formulations of some of the theorems that appear in \cite{AL},
in versions that are suitable to this work.
\newline
Let $G$ be an arbitrary group.
\begin{itemize}
\item  For any $G$-set $X$ and a point $x \in X$, we denote by $G(x)$ the
$G$-orbit of $x$ and by $G^x$ the stabilizer of $x$.
\item
For any representation of $G$ on a vector space $V$ and
a character $\chi$ of $G$, we denote by $V^{G,\chi}$ the subspace
of $(G,\chi)$-equivariant vectors in $V$.
\item
Given manifolds $L \subseteq M$, we denote by $N_L^{M} := \big(T_{M|L} \big)/T_L$ the normal
bundle to $L$ in $M$ and by $CN_L^M := (N_L^M)^*$ the conormal bundle. For any point
$y \in L$, we denote by $N_{L,y}^M$ the normal space to $L$ in $M$ at the point $y$
and by $CN_{L,y}^M$ the conormal space to $L$ in $M$ at the point $y$.
\item The symmetric algebra of a vector space $V$ will be denoted by
$$Sym(V) = \oplus_{k \geq 0} Sym^k(V).$$
\item The Fr\'echet space of Schwartz functions on a Nash manifold $X$ will be denoted by
$\cS(X)$ and the dual space of Schwartz distributions will be denoted by $\cS^*(X) := \cS(X)^*$.
\item For any Nash vector bundle $E$ over $X$ we denote by $\cS(X,E)$ the space of
Schwartz sections of $E$ and by $\cS^*(X,E)$ its dual space.
\end{itemize}
See \cite[p. 309]{AL} for more details. \newline
Recall that if $X$ is a smooth manifold and $G$ acts on $X$, then $X = \bigcup_{i=1}^{l} X_i$ is called
a $G$-invariant stratification if all sets $X_i$ are $G$-invariant and there
is some reordering $X_{i_1}, X_{i_2}, ..., X_{i_l}$ of $X_1,...,X_l$ such that all the sets
$X_{i_1}, X_{i_1} \cup X_{i_2}, ... , X_{i_1} \cup X_{i_2} \cup ... \cup X_{i_k} , ..., X = X_{i_1}\cup ... \cup X_{i_l}$ are open in $X$.
\begin{lem} \label{stratification}
Let a real algebraic Lie group $G$ act on a real algebraic smooth manifold $X$. Let
$X = \bigcup_{i=1}^{l} X_i$ be a $G$-invariant stratification. Let $\chi$
be a character of $G$. If $$\cS^*(X)^{G,\chi} \ne 0,$$ then there exists $1\le i \le l$ and
$k \ge 0$ such that $$\cS^*{(X_i, Sym^k(CN_{X_i}^X))^{G,\chi}} \ne 0.$$
\end{lem}
This lemma is a special case of \cite[Proposition B.3]{AL}.
\begin{thm}\label{Frobenius}\cite[Theorem B.6]{AL}
Let $G$ be a Lie group acting transitively on a smooth manifold $Z$ and
let $\varphi:X \to Z$ be a $G$-equivariant smooth map. Fix $z \in Z$ and let
$X_z$ be the fiber of $z$. Let $\chi$ be a tempered character of $G$ \cite[Definition 5.1.1]{AGS},
and let $\delta_{G}$ and $\delta_{G_{z}}$
be the modulus characters of the groups $G$ and $G_{z}$ respectively.
Then $S^*(X)^{G,\chi}$ is canonically isomorphic to
$S^*(X_z)^{G_z,\chi \delta_{G_z}^{-1} \delta_G}$. \newline
Moreover, for any $G$-equivariant bundle $E$ on $X$, the space
$S^*(X,E)^{G,\chi}$ is canonically isomorphic to $S^*(X_z,E|{X_z})^{G_z,\chi \delta_{G_z}^{-1} \delta_G }$.
\end{thm}
\end{section} 
%
\begin{section}{Some matrix spaces decompositions}
\newcommand{\tp}{^{t} \hspace{-1mm}}
In this section we obtain some matrix spaces decompositions that will be used in this work.
In the following lemma we analyze the structure of orbits of the action of
the Borel subgroup $B_n$ on
$X_n$. Let $W_n = S_n$ be the Weyl group of $G_n(\CC)$.
\begin{lem}\label{orbits classification}
There is a bijection between $B_n \backslash X_n = B_n \backslash G_n(\CC)/G_n(\RR)$ and the space of involutions
$W_{n,2} = \{w \in W_n: w^2 = 1\}$.
\end{lem}
\begin{proof}
Recall that $X_n = \{x \in G_n(\CC) | x \cdot \bar{x} = I\}$. Let $x \in X_n$.
Let $$ T = \{ diag(d_1,...,d_n) | \hspace{1mm} d_i \in \CC^*\ \text{ for all } i \}$$
be a maximal torus in $G_n(\CC)$. From \cite[Lemma 4.1.1]{LapRog}, see also \cite{Springer},
the $B_n$-orbit of $x$ intersects the normalizer $$N(T) := \{g \in G_n(\CC) | gTg^{-1} = T\}.$$
It is a well-known fact that $N(T) =\{ d \cdot w | d \in T, w \in W_n \}$.
Thus, we may assume $x = d w$, where $w\in W_n$ and $d=diag(d_1,d_2,...,d_n)$.
Note that $w$ is uniquely determined by $x$.
Since $x \cdot \bar{x} = I$, we have $d w =  w^{-1} \bar{d}^{-1}$. We obtain
$w=w^{-1}$ and therefore $w^2 =1$, i.e., $w \in W_{n,2}$. \newline Therefore, we can assume $w \in W_{n,2}$ in the decompostion $x= d w$. On the other hand, it is clear that different involutions $w,w' \in W_{n,2}$ belong to disjoint orbits
of $B_n$. Indeed, $l(b)w := bw\bar{b}^{-1} \neq w'$ for all $b \in B_n$. \newline
It remains to show that the $B_n$-orbit of $x= d w$ contains the point $w$, i.e.,
there is some $b \in B_n$ such that $\lambda(b)x = w$.
Since $w$ is an involution it is enough to check the claim for $1 \times 1$ and
$2 \times 2$ matrix. For $1\times 1$ matrix $x=(b)_{1\times 1}$, the assumption $x\bar{x} = I$ gives
$b\bar{b} = 1$, and we want to prove that $b = \mu \bar{\mu}^{-1}$. Clearly, there is such $\mu$. \newline
For a $2\times 2$ matrix of the form
$ b = \left( \begin{array}{cc} d_1 & 0 \\ 0 & d_2 \\ \end{array} \right)$,
the assumption
$$ x = b \left(\begin{array}{cc} 0 & 1 \\ 1 & 0 \\ \end{array} \right)
= \left( \begin{array}{cc} 0 & d_1 \\ d_2 & 0 \\ \end{array} \right) \in X $$
gives the condition $d_1\bar{d_2} = 1$ on the entries $d_1,d_2$. We seek for
an invertible matrix
$\left(\begin{array}{cc}\mu_1 & 0 \\0 & \mu_2 \\\end{array}\right)$
such that
\begin{equation}\label{matrix identity}
\left(\begin{array}{cc} \mu_1 & 0 \\ 0 & \mu_2 \\ \end{array} \right)
\left(\begin{array}{cc} 0 & d_1 \\ d_2 & 0 \\ \end{array} \right)
\left(\begin{array}{cc}\bar{\mu_1}^{-1} & 0 \\0 & \bar{\mu_2}^{-1} \\\end{array}\right) =
\left(\begin{array}{cc} 0 & 1 \\ 1 & 0 \\ \end{array} \right).
\end{equation}
Matrix multiplication gives the condition $d_1 \mu_1 \bar{\mu_2}^{-1} = 1$, and clearly there
are such $\mu_1, \mu_2$.
\end{proof}
In the next paragraph let us fix $n$ and denote $G = G_n(\CC), H = G_n(\RR)$.
Our next goal is to obtain a generalized Cartan decomposition $G = K A H$, where
$K$ is a maximal compact subgroup of $G$ consisting of all unitary matrices in $G$ and $A$ is a torus which we will now describe.
Let $m = [n/2]$.
Note that $H = G^{\sigma}$ and $K = G^{\tau}$, where $\sigma(g) = \bar{g}$ and
$\tau(g) = g^* = ^{t} \hspace{-1mm} \bar{g}$. Let $\gotG$ be the Lie algebra of $G$ over the field $\CC$.
Following \cite[Fact 2.1,page 7]{Kob}, we take
$\mathfrak{a}$ to be a maximal abelian subspace in
$$\gotG^{-\sigma,-\tau} = \{ X \in \gotG : \tau X = \sigma X = - X\}. $$
Following this recipe, let us define $$\mathfrak{a} = \sum_{j=1}^m{i\RR(E_{2j+1,2j} - E_{2j,2j+1}}).$$
Recall that $$exp\left(\begin{array}{cc}
                                0 & it \\
                                -it & 0
                              \end{array}
  \right) = \left(\begin{array}{cc}
                                ch(t) & ish(t) \\
                                -ish(t) & ch(t)
                              \end{array}
  \right)
   $$ and
let $A$ be the Lie group corresponding to $\mathfrak{a}$.
Denote by $a(t_{1},t_{2},...,t_{m})$
  the $n\times n$
  matrix which consists of $m$
  $2\times2$ diagonal blocks of the form $exp\left(\begin{array}{cc}
0 & it_{j}\\
-it_{j} & 0
\end{array}\right)$
  , where $j=1,2,...,m$
  if $n=2m$
  is even, and which consists of these blocks and $a_{nn}=1$
  in the last diagonal place if $n=2m+1$
  is odd. For example, if $n=4$
  then $$a(t_{1},t_{2})=\left(\begin{array}{cccc}
\cosh(t_{1}) & i\sinh(t_{1}) & 0 & 0\\
-i\sinh(t_{1}) & \cosh(t_{1}) & 0 & 0\\
0 & 0 & \cosh(t_{2}) & i\sinh(t_{2})\\
0 & 0 & -i\sinh(t_{2}) & \cosh(t_{2})
\end{array}\right).$$
We have $$A=\{ a(t_1,t_2,...,t_m): t_1,t_2, ...,t_m \in \RR \}.$$
Define
$$A^+ = \{ a(t_1,t_2,...,t_m): t_1 \ge t_2 \ge \hdots \ge t_m \ge 0\}. $$
\begin{thm}\label{thm: KAH decomposition}
There is a decomposition $G = K A^+ H$ , that is, every element $g \in G$ can be written as
\begin{equation}\label{eq: KAH}
 g = kah, \text{ where }  k \in K,  a \in A^+, h \in H.
\end{equation}
Moreover,the $a$ part in decomposition (\ref{eq: KAH}) is uniquely determined by $g$.
\end{thm}
\begin{rem}
By taking transpose on \eqref{eq: KAH} we obtain a similar decomposition
$G = HA^+K$, that is, every $g \in G$ can be written as
\begin{equation}\label{eq: HAK}
 g = hak, \text{ where }  h \in H,  a \in A^+, k \in K,
\end{equation}
and $a \in A^{+}$ in this decomposition is uniquely determined by $g$. Actually, after taking transpose on \eqref{eq: KAH}, we
obtain that $a^t \in A$ and in general $a^t \notin A^+$. But the permutation group $S_n$ is naturally contained in both
$K$ and $H$ and we can replace $a \in A$ with $a' = w_1aw_2$ such that $a' \in A^+$.
\end{rem}
\begin{proof}[Proof of Theorem \ref{thm: KAH decomposition}]
To prove the existence part we will show that $G = KAH$. Since permutation matrices are clearly in
$H \cap K$, the equality $G = KA^+H$ will easily follow from the equality $G=KAH$.
Let $g \in G$. Our goal is to achieve a decomposition
$g =k a h$ with $ h \in H$, $ a \in A$, and $k \in K$. Suppose that $g$ is of such form. Then, since
$h^* = \tp h, a^* = a,$ and $k^* = k^{-1}$ we get
\begin{equation}\label{eq:Cartan decomp}
g^* g = \tp h a^2 h, h\in H, a \in A.
\end{equation}
On the other hand, suppose that every matrix of the form
$g^* g$ can be written as (\ref{eq:Cartan decomp}). Then write
$ g = ((g^*)^{-1} \ ^t  h a) a h$. Let's show that $k = (g^*)^{-1} \ ^t h a$ is a unitary matrix.
Indeed, $$k^* k = a^* (^t h)^* ((g^*)^{-1})^* (g^*)^{-1} \ ^t h a = a h (g^* g)^{-1}\ ^t h a = a h (^t h a^2h)^{-1} \ ^t h a =I.$$
Therefore, to prove the existence part in the theorem,
it is enough to prove that every matrix of the form $g^*g$ can be written
in the form (\ref{eq:Cartan decomp}).
For this purpose write $g^* g = x + i y, x,y \in H$. Then $x= \hspace{1mm} ^t x$ is symmetric, and $y= -\tp y$ is antisymmetric.
Also, $^t v g^* g v > 0$ for every $0\neq v \in \RR^n$. Hence,
$x$ is a positive definite matrix, that is, $^tv x (^t x) v  = \tp v g^* g v >0 $ for every $0 \neq v \in \RR^n$.
Thus, there is
a matrix $h \in H$ such that $^t h x h = I$. Then $^t h g^* g h = I + i (^th yh)$. The matrix $z := \tp hyh$ is
antisymmetric and it is a standard fact in linear algebra that it is diagonalizable by a real orthogonal matrix.
Consequently, $h' z h'^{-1} = d$, with $d$ consisting of $m = \lfloor n/2 \rfloor$ $2 \times 2$ blocks of the
form $$\left(
        \begin{array}{cc}
          d_{2i-1,2i-1} & d_{2i-1,2i} \\
          d_{2i,2i-1} & d_{2i,2i} \\
        \end{array}
      \right) =
\left(
        \begin{array}{cc}
          0 & \lambda_i \\
          -\lambda_i & 0 \\
        \end{array}
      \right)
 $$  in the case $n$ even, and $m$ such blocks and the last zero row in the case $n$ odd.
 Note also that the numbers $\lambda_i$ are uniquely determined up to a permutation by
 the matrix $gg^*$ since they are eigenvalues of
 $hy ^th$.
 Clearly, every block $2 \times 2$ of the form $\left(
        \begin{array}{cc}
          1 & i \lambda \\
          -i \lambda & 1 \\
        \end{array}
      \right) $ can be transformed by a diagonal matrix $
      \left(
        \begin{array}{cc}
          d_1 & 0 \\
          0  & d_2 \\
        \end{array}
      \right)
       $ to the form
 $$\left(
 \begin{array}{cc}
          ch(\mu) & i sh(\mu) \\
          -i sh(\mu) & ch(\mu) \\
        \end{array}
      \right) = exp\left( \mu \left(\begin{array}{cc}
          0 & i  \\
          -i  & 0 \\
        \end{array} \right) \right)
 $$
 Taking in every block $a$ of the form
 $exp\left( \frac{\mu}{2} \left(\begin{array}{cc}
          0 & i  \\
          -i  & 0 \\
        \end{array} \right) \right) $
proves the existence of the decomposition $g^*g = \tp h a^2 h$ and thus
establishes the existence of the decomposition $g = k a h$.\\
Uniqueness:
note that $H$ acts on the space of positive definite matrices
of the form $g^*g$ by $h \cdot x := h^t x h$. Let us take $h,b,c \in H$ and suppose
$h \cdot(I + ib) = I + ic$. Then $h$ is an orthogonal matrix, $^th h = I$, thus
$c = h^{-1}bh$. In particular, the eigevenvalues of $b$ and $c$ are equal.
Now, to prove the uniqueness of $a \in A^+$
in the decomposition (\ref{eq: KAH}) let us write
$a = Re(a) + i Im(a)$ and note that $H \cdot a = H \cdot (I + i Im(a))$.
Since eigenvalues of $iIm(a)$ are $\pm \sinh(\lambda_1), ...,\pm \sinh(\lambda_n)$ we
see that if $a_1,a_2 \in A^+$ and $a_1 \neq a_2$, then
$H \cdot (I + i Im(a_1)) \neq H \cdot (I + iIm(a_2)),$ and therefore
$H \cdot a_1 \neq H \cdot a_2$. It follows that the $a^2 \in A^+$ part in $g^*g = \tp h a^2 h$ is uniquely
determined by $g$. As a result, $a \in A^+$ is uniquely determined by $g$.
\end{proof}

\end{section} 
%
\begin{section}{Proof of Theorem \ref{thm:induced representation}}
In this paragraph $n$ is fixed, $G=G_{n}(\CC),H=G_{n}(\RR),$ and
$B=B_{n}(\CC)$. We denote by $M$ the standard maximal torus in $G$
and by $W_{2}=W_{n,2}$ the set of involutions in $S_{n}$. As a starting
point of the proof, observe that
\[
I(\chi)^{*}=\cS^{*}(G)^{B,\chi\delta_{0}^{-\frac{1}{2}}},
\]
where $B$ acts on the space of tempered distributions $\cS^{*}(G)$
from the left. We have
\[
\Hom_{H}(I(\chi),\CC)=\cS^{*}(G/H)^{B,\chi\delta_{0}^{-\frac{1}{2}}}.
\]
We will stratify $X:=G/H$ by $B$-orbits. By Lemma \ref{orbits classification},
we have $B\backslash X=W_{2}$. Suppose $\Hom_{H}(I(\chi),\CC)\ne 0$.
By Lemma \ref{stratification} there exists an involution $w\in W_{2}$
and $k\ge0$ such that
\[
\cS^{*}\big(B(w),Sym^{k}(CN_{B(w)}^{X})\big)^{B,\chi\delta_{0}^{-\frac{1}{2}}}\ne0.
\]
Note that $B$ acts on $B(w)$ transitively. Note that the stabilizer
of $w$ under the action of $B$ is $B_{w}$ and $\delta_{0}^{\frac{1}{2}}|_{B^{w}}=\delta_{B^{w}}.$
Therefore, by Frobenius reciprocity (Theorem \ref{Frobenius}),
\[
\begin{aligned}\cS^{*}\big(B(w),Sym^{k}(CN_{B(w)}^{X})\big)^{B,\chi\delta_{0}^{-\frac{1}{2}}} & =\cS^{*}\big(\{w\},Sym^{k}(CN_{B(w),w}^{X})\big)^{B^{w},\chi\delta_{0}^{-\frac{1}{2}}\delta_{B^{w}}^{-1}\delta_{0}}\\
 & =\cS^{*}\big(\{w\},Sym^{k}(CN_{B(w),w}^{X})\big)^{B^{w},\chi}\\
 & =\big(Sym^{k}(N_{B(w),w}^{X})\otimes_{\RR}\CC\big)^{B^{w},\chi}.
\end{aligned}
\]
Observe that $M^{w}\subset B^{w}$, hence $\big(Sym^{k}(N_{B(w),w}^{X})\otimes_{\RR}\CC\big)^{B^{w},\chi}\ne0$
implies $$\big(Sym^{k}(N_{B(w),w}^{X})\otimes_{\RR}\CC\big)^{M^{w},\chi}\ne0.$$
Note that
\[
\begin{aligned}M^{w}&=\{t\in M:t^{-1}w\bar{t}=w\}=\{t\in M:t=w\bar{t}w\}  \\
&=\{t=\text{diag}(t_{1},t_{2},...t_{n})\in M:t_{i}=\overline{t_{w(i)}}\text{ for }1\leq i\leq n\}.
\end{aligned}
\]
It will be useful to obtain one more formula for the $M^{w}$. It
is easy to see, by examining case of $1\times1$ and $2\times2$ matrices,
that
\begin{equation}
M^{w}=\left\{ t(w\bar{t}w)a|t\in M,a=\text{diag}(a_{1},a_{2},...,a_{n}),a_{i}=1\text{ if }w(i)\neq i,a_{i}=\pm1\text{ if }w(i)=i\right\} .\label{eq:M^w as image}
\end{equation}
In the next lemma we perform a calculation of the normal space $N_{B(w),w}^{X}$.
Note that this is a finite-dimensional vector space over $\RR$. Since
the group $M^{w}$ preserves the tangent space $T_{w}^{B(w)}$ and
clearly preserves the tangent space $T_{w}^{X}$, there is an action
of $M^{w}$ on the normal space $N_{B(w),w}^{X}$. By taking the scalar
extension $N_{B(w),w}^{X}\otimes\CC$, we get a complex
representation of $M^{w}$. Since $M^{w}$ is abelian, this representation
decomposes into a direct sum of irreducible, one-dimensional representations.
\begin{lem}
\label{lem:normal sp repr}We have
\[
N_{B(w),w}^{X}\otimes_{\RR}\CC=\bigoplus_{\{(i,j)\in I_{w}\}}\alpha_{\delta(i,j)}
\]
as a representation of $M^{w}$.
\end{lem}
Before proving this lemma, we give the following corollary.
\begin{cor}
We have
\[
Sym(N_{B(w),w}^{X}\otimes_{\RR}\CC)=\bigoplus_{\kappa:I_{w}\to\ZZ_{\geq0}}\alpha_{\kappa}
\]
as a representation of $M^{w}$.\global\long\def\MnC{\mathrm{Mat}_{n}(\CC)}
\end{cor}
\begin{proof}[Proof of Lemma \ref{lem:normal sp repr}] Let us denote by
$e_{i,j}$ the elementary matrix with $1$ at the $(i,j)$-th entry
and zeros in all other entries. The tangent space of $X$ at $w$
is equal to
\[
\begin{split}T_{w}^{X} & =\{A\in\MnC|Aw+w\bar{A}=0\}=\{A\in\MnC|wAw=-\bar{A}\}\\
&=\text{Span}_{\RR}\{-e_{i,j}+e_{w(i),w(j)},\sqrt{-1}(e_{i,j}+e_{w(i),w(j)})|1\leq i,j\leq n\}.
\end{split}
\]
On the other hand,
\[
T_{w}^{B(w)}=\{-Aw+w\bar{A}|A\in\MnC,\, A\text{ is upper triangular}\}
\]
Since $e_{i,j}w=e_{i,w(j)}$ and $we_{i,j}=e_{w(i),j}$, we obtain
that
\[
\begin{split}T_{w}^{B(w)} & =\text{Span}_{\RR}\{-e_{i,w(j)}+e_{w(i),j},\sqrt{-1}(e_{i,w(j)}+e_{w(i),j})|1\leq i\leq j\leq n\}\\
 & =\text{Span}_{\RR}\{-e_{i,j}+e_{w(i),w(j)},\sqrt{-1}(e_{i,j}+e_{w(i),w(j)})|\text{ }i\leq w(j)\}\\
 & =\text{Span}_{\CC}\{e_{i,j},e_{w(i),w(j)}|\text{ }i\leq w(j)\}\cap T_{w}^{X}\\
 & =\text{Span}_{\CC}\{e_{i,j}|\text{ }i\leq w(j)\text{ or }j\geq w(i)\}\cap T_{w}^{X}.
\end{split}
\]
Hence
\begin{equation}
\begin{split}N_{B(w),w}^{X} & \cong\text{Span}_{\CC}\{e_{i,j}|\text{ }i>w(j),w(i)>j\}\cap T_{w}^{X}\\
 & =\text{Span}_{\CC}\{e_{i,w(j)}|\text{ }i>j,w(i)>w(j)\}\cap T_{w}^{X}\\
 & =\text{Span}_{\CC}\{e_{i,w(j)}|\text{ }(i,j)\in I_{w}\}\cap T_{w}^{X}.
\end{split}
\label{eq:normal space complex}
\end{equation}
Let us denote $V=\text{Span}_{\CC}\{e_{i,w(j)}|\text{ }(i,j)\in I_{w}\}$.
Note that if $e_{i,w(j)}\in V$ then also $e_{w(i),j}\in V,$ since
$w$ is an involution and for an involution
\[
(i,j)\in I_{w}\iff(w(i),w(j))\in I_{w}.
\]
Let us use the lexicographic ordering on pairs $(i,j)$ : write $(i,j)<(i',j')$
if $i<i'$ or ($i=i'$ and $j<j'$) . Then we may rewrite (\ref{eq:normal space complex})
as
\begin{equation}
\begin{split}N_{B(w),w}^{X} & \cong\text{Span}_{\RR}\{\sqrt{-1}e_{i,w(j)}|\text{ }(i,j)\in I_{w},(i,j)=(w(i),w(j))\}\\
 & \oplus\text{Span}_{\RR}\{e_{i,w(j)}-e_{w(i),j},\sqrt{-1}(e_{i,w(j)}+e_{w(i),j})|\text{ }(i,j)\in I_{w},(i,j)<(w(i),w(j))\}.
\end{split}
\label{eq: normal sp real}
\end{equation}
For $t=\text{diag}(t_{1},...,t_{n})\in M$ we have
\[
te_{i,j}\bar{t}^{-1}=(t_{i}/\bar{t_{j}})e_{i,j},
\]
and for $t\in M^{w}$ we also have
\[
te_{i,w(j)}\bar{t}^{-1}=(t_{i}/\overline{t_{w(j)}})e_{i,w(j)}=(t_{i}/t_{j})e_{i,w(j)}.
\]
Therefore the action of $M^{w}$ on $e_{i,w(j)}$ is given by $\alpha_{\delta(i,j)}$.
We obtain that $N_{B(w),w}^{X}\otimes_{\RR}\CC$, as a representation
of $M^{w}$ is
\[
\begin{split}\begin{split}N_{B(w),w}^{X}\otimes_{\RR}\CC\end{split}
 & \cong\bigoplus_{\{(i,j)\in I_{w}:(i,j)=(w(i),w(j))\}}\alpha_{\delta(i,j)}\oplus\bigoplus_{\{(i,j)\in I_{w}:(i,j)<(w(i),w(j))\}}(\alpha_{\delta(i,j)}\oplus\alpha_{\delta(w(i),w(j))})\\
 & = \bigoplus_{\{(i,j)\in I_{w}:(i,j)=(w(i),w(j))\}}\alpha_{\delta(i,j)}\oplus\\
 &  \bigoplus_{\{(i,j)\in I_{w}:(i,j)<(w(i),w(j))\}}\alpha_{\delta(i,j)}\oplus\bigoplus_{\{(i,j)\in I_{w}:(i,j)>(w(i),w(j))\}}\alpha_{\delta(i,j)}\\
 & =\bigoplus_{\{(i,j)\in I_{w}\}}\alpha_{\delta(i,j)}.
\end{split}
\]
The last equality is exactly the assertion of the lemma.
\end{proof}
\begin{lem}
\label{lem:dim est} If $\left(Sym^{k}(N_{B(w),w}^{X})\otimes_{\RR}\CC\right)^{M^{w},\chi}\ne0$
then $k=0$, $w(\chi)=\bar{\chi}^{-1}$, and $\chi_{i}(-1)=1$ for
all $1\le i\le n$ such that $w(i)=i.$
\end{lem}
\begin{proof}
Note that for $t\in M$ and $w\in W_{2}$, the element $wtw$ is also
diagonal and its diagonal entries are the permutation of diagonal
entries of $t$ by $w$, i.e.,
\[
(wtw)_{ii}=t_{w(i),w(i)}.
\]
We are going to use (\ref{eq:M^w as image}). From it follows that
if $\alpha_{\kappa}|_{M^{w}}=\chi|_{M^{w}}$, then for every $t\in M$
we have
\[
\alpha_{\kappa}(t(w\bar{t}w))=\alpha_{\kappa}(t)\overline{w(\alpha_{\kappa})(t)}=\chi(t)\overline{\chi(w(t))}=\chi(t(w\bar{t}w)).
\]
That is,
\begin{equation}
\big(\alpha_{\kappa}|_{M^{w}}=\chi|_{M^{w}}\big)\Rightarrow\big(\overline{\alpha_{\kappa}}w(\alpha_{\kappa})=\bar{\chi}w(\chi)\big).\label{eq:cond S_w}
\end{equation}
To obtain (\ref{eq:cond S_w}) just put $a=1$ in (\ref{eq:M^w as image}).
The set of $\kappa$-s that satisfy
\begin{equation}
\overline{\alpha_{\kappa}}w(\alpha_{\kappa})=\bar{\chi}w(\chi)\label{eq:kappa_eq}
\end{equation}
is $\{\kappa\equiv0\}$ if $w(\chi)=\overline{\chi^{-1}}$ and is
empty otherwise. Indeed, let us take absolute values on two sides
of (\ref{eq:kappa_eq}). We obtain
\begin{equation}
\prod_{(i,j)\in I_{w}}\big|\frac{t_{i}}{t_{j}}\big|^{\kappa(i,j)+\kappa(w(i),w(j))}=\prod_{i=1}^{n}|t_{i}|^{\lambda_{i}+\lambda_{w(i)}}.\label{eq:abs_kappa}
\end{equation}
First, we will deduce from the last equation that the right hand-side
of this equation is $1$. Note that from (\ref{eq:abs_kappa}) follows,
by substituting $t_{i}=c$ for all $i$ with a generic $c\in\CC^{*}$,
that
\[
\lambda_{1}+...+\lambda_{n}=0.
\]
Since no pair $(1,i)$ belongs to $I_{w}$, it follows that $\lambda_{1}+\lambda_{w(1)}\leq0$.
Let $i$ be the first index such that $\lambda_{i}+\lambda_{w(i)}>0$.
Then on the left hand side of (\ref{eq:abs_kappa}) the power of $|t_{i}|$
is positive, thus there is a $j$ such that $(i,j)\in I_{w}$. Hence
$i>j,w(i)>w(j)$ and from the assumption $\lambda_{1}\ge\lambda_{2}\ge...\ge\lambda_{n}$
we obtain $\lambda_{i}\leq\lambda_{j}$ and $\lambda_{w(i)}\leq\lambda_{w(j)}$.
Thus
\[
0<\lambda_{i}+\lambda_{w(i)}\leq\lambda_{j}+\lambda_{w(j)}\leq0,
\]
a contradiction! Therefore, for every $i$, there is an inequality
$\lambda_{i}+\lambda_{w(i)}\leq0$. Since the sum of all $\lambda$'s
is equal to $0$, we obtain $\lambda_{i}+\lambda_{w(i)}=0$ for every
$i$. Hence, $\lambda_{1}\geq0,\lambda_{n}\leq0$.\\
Now, we can deduce $\kappa\equiv0$. Indeed, let $j$ be the minimal
index such that there exists $(i,j)\in I_{w}$ with the property
\[
\kappa(i,j)\neq0\text{ or }\kappa(w(i),w(j))\neq0.
\]
The power of $|t_{j}|$ on the right hand side of (\ref{eq:kappa_eq})
must equal $0$, thus there is a pair $(j,k)\in I_{w}$ such that
$\kappa(j,k)\ne0$ or $\kappa(w(j),w(k))\ne0.$ In both cases we obtain
a contradiction to the minimality of $j$. As a conclusion, we obtain
that (\ref{eq:kappa_eq}) implies $\kappa\equiv0$. \\
Suppose now $w(\chi)=\overline{\chi^{-1}}$ and thus $\kappa\equiv0$.
Then $\alpha_{\kappa}=1$, the identity character. We want to prove
that $\chi_{i}(-1)=1$ for all $i$ such that $w(i)=i$. This follows
from $\chi(a)=\alpha_{\kappa}(a)=1$ for $a=diag(a_{1},...,a_{n})$,
where $a_{i}=\pm1$ whenever $w(i)=i$ and $a_{i}=1$ otherwise.
\end{proof}
 \end{section}
%
\begin{section}{Calculation of Rankin-Selberg Gamma factors}\label{section: Rankin-Selberg}
In this section we recall the notion of Rankin-Selberg integrals and
apply the results of previous sections to calculate special values
of Rankin-Selberg gamma factors. The exposition and notation follows
\cite{Jacquet}. Let $\chi:B_{n}\to\CC^{\times}$
be a multilicative character and let $\lambda:I(\chi)\to\CC$ be a
$\psi$-form on $I(\chi)$. Recall that such $\lambda$ always exists
and it is unique up to a scalar multiple. If $v\in V$ and $f\in V$,
$g\in G_{n}$, we set
\[
W_{f}(g)=\lambda(R(g)f).
\]
Let $\cW(I(\chi),\psi)$ be the space spanned by the functions of
the form $W_{f}$, where $f\in V$. \\
For every $n$, we denote by $w_{n}$ the $n\times n$ permutation
matrix whose anti diagonal entries are $1$. If $n>n',$ we denote
\[
w_{n,n'}=\left(\begin{array}{cc}
1_{n'} & 0\\
0 & w_{n-n'}
\end{array}\right).
\]
 \\
If $f\in I(\chi)$, then the function $\tilde{f}$ is defined by
\[
\tilde{f}(g):=f(w_{n}\hspace{1mm}^{t}g^{-1}).
\]
Let $\pi$ be an irreducible representation of $G_{n}(\CC)$ and let
$\pi'$ be an irreducible representation of $G_{m}(\CC)$. Suppose
$\pi$ is the Langlands quotient of $I(\chi)$ and $\pi'$ is the
Langlands quotient of $I(\chi')$. We choose a $\psi$-form $\lambda$
on $I(\chi)$ and a $\bar{\psi}$-form $\lambda'$ on $I(\chi')$.
Rankin-Selberg integrals are defined as follows. For $f\in I(\chi),f'\in I(\chi')$,
set
\[
W=W_{f},\, W'=W_{f'}.
\]
For $W=W_{f}$, set
\[
\widetilde{W}_{f}:=W_{f}(w_{n}\hspace{1mm}^{t}g^{-1}).
\]
Note that $\widetilde{W}_{f}(g)=W_{\tilde{f}}(g)$ and $W_{\tilde{f}}(g)\in\cW(I(\chi^{-1}),\bar{\psi})$.\\
If $n>n'$ we set
\begin{equation}
\Psi(s,W,W')=\int W\left(\begin{array}{cc}
g & 0\\
0 & 1_{n-n'}
\end{array}\right)W'(g)|\det g|_{\CC}^{s-\frac{n-n'}{2}}dg.\label{eq:RS1}
\end{equation}
In addition, for $0\le j\le n-n'-1$, we set
\begin{equation}
\Psi_{j}(s,W,W')=\int W\left(\begin{array}{ccc}
g & 0 & 0\\
X & 1_{j} & 0\\
0 & 0 & 1_{n-n'-j}
\end{array}\right)W'(g)|\det g|_{\CC}^{s-\frac{n-n'}{2}}dXdg.\label{eq:RS2}
\end{equation}
Here $X$ is integrated over the space $M(m\times j,\CC)$ of matrices
with $m$ rows and $j$ columns. In each integral,$g$ is integrated
over the quotient $U_{n}(\CC)\backslash G_{n}(\CC).$\\
If $n=n'$, we let $\Phi$ be a Schwartz function on $\CC^{n}$ and
we set
\begin{equation}
\Psi(s,W,W',\Phi)=\int W(g)W'(g)\Phi[(0,0,...,0,1)g]|\det g|_{\CC}^{s}dg.\label{eq:RS3}
\end{equation}
Rankin-Selberg gamma factor $\gamma(s,\pi\times\pi',\psi)$ is a proportionality
factor appearing in functional equation on certain Rankin-Selberg
integrals. We quote here \cite[Theorem 2.1]{Jacquet}.
\begin{thm}
\label{thm: Jacquet func_eqn}
\begin{enumerate}
\item The integrals (\ref{eq:RS1}),(\ref{eq:RS2}), and (\ref{eq:RS3})
converge for $Re(s)>>0$.
\item Each integral extends to a meromorphic function of $s$ which is a
holomorphic multiple of $L(s,\pi\times\pi')$ bounded at infinity
in vertical strips. See \cite{Jacquet} for the definition of $L(s,\pi\times\pi').$
\item The following functional equations are satisfied. If $n>n'$,
\[
\Psi(1-s,\widetilde{W},\widetilde{W'})=\omega_{I(\chi)}(-1)^{n-1}\omega_{I(\chi')}(-1)\gamma(s,I(\chi)\times I(\chi'),\psi)\Psi(s,W,W').
\]
If $n>n'+1$,
\[
\Psi_{j}(1-s,\pi(w_{n,n'})\widetilde{W},\widetilde{W'})=\omega_{I(\chi)}(-1)^{n'}\omega_{I(\chi')}(-1)\gamma(s,\pi\times\pi',\psi)\Psi_{n-n'-1-j}(s,W,W').
\]
If $n=n'$,
\[
\Psi(1-s,\widetilde{W},\widetilde{W'},\hat{\Phi})=\omega_{I(\chi)}(-1)^{n-1}\gamma(s,\pi\times\pi',\psi)\Psi(s,W,W',\Phi).
\]

\end{enumerate}
\end{thm}
We will calculate the special values of Rankin-Selberg gamma factor
of $G_{n}(\RR)$-distinguished representations. The main tool will
be the classification of such representations obtained in Theorem
\ref{thm:induced representation} and basic properties of Rankin-Selberg
gamma factors, \cite[Lemma 16.3]{Jacquet}.
\begin{thm}
\label{thm: Rankin-Selberg} Let $\pi$ be an irreducible $G_{t}(\RR)$-distinguished
representation of $G_{t}(\CC)$ and $\pi'$ be an irreducible and
$G_{r}(\RR)$-distinguished representation of $G_{r}(\CC)$. If $\psi$
is a non-trivial character of $\CC$ with a trivial restriction to
$\RR$ then
\[
\gamma\left(\frac{1}{2},\pi\times\pi',\psi\right)=1.
\]

\end{thm}
A similar theorem is proven in \cite[Theorem 0.1]{O}
for the $p$-adic case.\\
Before proving the theorem, let us recall some basic facts about one-dimensional
Tate gamma factors. Let $\chi$ be a one-dimensional character $\chi:\CC^{*}\to\CC^{*}$.
We have the following functional equation for Tate gamma-factors
\begin{equation}
\gamma(s,\chi,\psi)\gamma(1-s,\overline{\chi^{-1}},\overline{\psi^{-1}})=1.\label{eq:gamma func eq}
\end{equation}
Since we assume $\psi$ is trivial on $\RR$, we obtain $\overline{\psi^{-1}}=\psi$,
and for $s=\frac{1}{2}$ we get
\begin{equation}
\gamma\left(\frac{1}{2},\chi,\psi\right)\gamma\left(\frac{1}{2},\overline{\chi^{-1}},\psi\right)=1.\label{eq:pair of gammas at half}
\end{equation}
For a real character $\chi$, that is, for $\chi$ satisfying $\chi^{2}=1$,
we obtain $\gamma\left(\frac{1}{2},\chi,\psi\right)^{2}=1$, and thus
$\gamma\left(\frac{1}{2},\chi,\psi\right)\in\{1,-1\}$. The value
of $\gamma\left(\frac{1}{2},\chi,\psi\right)$ depends on $\chi(-1)$.
Whenever $\chi(-1)=1$ we obtain
\begin{equation}
\gamma\left(\frac{1}{2},\chi,\psi\right)=1\label{eq:real gamma factor at half}
\end{equation}

\begin{proof}[Proof of Theorem \ref{thm: Rankin-Selberg}]
Recall that if $\pi$ is the Langlands quotient of $Ind(\chi)$ and
$\pi'$ is the Langlands quotient of $Ind(\chi')$, then
\[
\gamma(s,\pi\times\pi',\psi)=\gamma(s,Ind(\chi)\times Ind(\chi'),\psi).
\]
It is well-known that $\chi=(\chi_{1},...,\chi_{t})$ , where $\chi_{i}$'s
are one-dimensional characters of $\CC$. Similarly, $\chi'=(\chi_{1}',...,\chi_{r}')$,
where $\chi_{i}'$'s are one-dimensional characters of $\CC$. Thus,
\begin{equation}
\gamma(s,Ind(\chi)\times Ind(\chi'),\psi)=\prod_{i=1}^{t}\gamma(s,\chi_{i}\times Ind(\chi'),\psi)=\prod_{i=1}^{t}\prod_{j=1}^{r}\gamma(s,\chi_{i}\chi_{j}',\psi).\label{eq:one_dim_gamma}
\end{equation}
Using Theorem \textbackslash{}ref\{thm:induced representation\}, there
exists an involution $w\in S_{t}$ and an involution $w'\in S_{r}$
such that $w(\chi)=\bar{\chi}^{-1}$ and $w'(\chi')=\bar{\chi'}^{-1}$
and for every fixed point $i$ of $w$, and $j$ of $w'$, we have
$\chi_{i}(-1)=1$ and $\chi_{j}'(-1)=1$. The formula in \ref{eq:one_dim_gamma}
may be rewritten as
\[
\gamma(s,Ind(\chi)\times Ind(\chi'),\psi)=I_{1}I_{2},
\]
where
\[
I_{1}=\prod_{\{(i,j):(w(i),w'(j))=(i,j)\}}\gamma\left(\frac{1}{2},\chi_{i}\chi_{j}',\psi\right),
\]
\[
I_{2}=\prod_{\{(i,j):\, i<w(i)\,\text{{\,\ or\,}\,}(i=w(i)\,\text{{\,\ and\,}}w'(j)<j)\}}\gamma\left(\frac{1}{2},\chi_{i}\chi_{j}',\psi\right)\gamma\left(\frac{1}{2},\chi_{w(i)}\chi_{w'(j)}',\psi\right).
\]
Let us prove that every term appearing in the product $I_{1}$ is
$1$. Indeed, by Theorem \ref{thm:induced representation} the character
$\chi_{i}\chi_{j}'$ appearing as the argument in the gamma factor
in $I_{1}$ is a real character satisfying $\chi\chi_{j}'(-1)=1$,
and therefore, by (\ref{eq:real gamma factor at half}) we get
\[
\gamma\left(\frac{1}{2},\chi_{i}\chi_{j}',\psi\right)=1.
\]
Each term in the product $I_{2}$ is also equals $1$, since $\chi_{w(i)}\chi_{w'(j)}'=(\overline{\chi_{i}\chi_{j}'})^{-1}$
and by applying (\ref{eq:pair of gammas at half}). Finally, $I_{1}=I_{2}=1$
and we obtain
\[
\gamma\left(\frac{1}{2},\pi\times\pi',\psi\right)=1.
\]

\end{proof}
We will need the following technical result about Rankin-Selberg integrals
in Section \ref{sec: functionals}.
\begin{lem}
\label{lem: RS convergence}Let $(\pi,V),\,(\pi',V')$ be generic
representations of $G_{n}(\CC)$ and let $$\mathcal{W}(\pi,\psi),\, \mathcal{W}(\pi',\psi^{-1})$$
be their Whittaker models. Suppose $(\pi,V)$ is unitarizable and
$(\pi',V')$ is tempered. Let
$W\in\mathcal{W}(\pi,\psi)),\, W'\in\mathcal{W}(\pi',\psi^{-1}),$
 and let $\Phi\in\mathcal{S}(\CC^{n})$ . Then the Rankin-Selberg
integral
\[
\int\limits _{U_{n}(\CC)\backslash G_{n}(\CC)}W(g)W'(g)\Phi((0,0,...,0,1)g)|\det g|_{\CC}^{s}dg
\]
 converges absolutely at $s=\frac{1}{2}$.\end{lem}
\begin{proof}
Let $T_{n}$ be the standard maximal torus in $G_{n}(\CC)$ and let
$K_{n}$ be a maximal compact subgroup of $G_{n}(\CC)$ consisting
of all unitary matrices in $G_{n}(\CC)$. Let $\delta$ be the modular
character of $B_{n}(\CC)$. By \cite[Lemma 2.1]{LapMao},
we know that there exists $\lambda>-\frac{1}{2}$ and $\phi\in\cS(\CC^{n-1})$
such that
\[
|W(tk)|\le\delta^{\frac{1}{2}}(t)|\det t|_{\CC}^{\lambda}|t_{n}|_{\CC}^{-n\lambda}\phi\left(\frac{t_{1}}{t_{2}},...,\frac{t_{n-1}}{t_{n}}\right)
\]
 for $t\in T_{n},k\in K_{n}$. Similarly, there exists $\phi'\in\cS(\CC^{n-1})$
such that
\[
|W'(tk)|\le\delta^{\frac{1}{2}}(t)\phi'\left(\frac{t_{1}}{t_{2}},...,\frac{t_{n-1}}{t_{n}}\right)
\]
for all $t\in T_{n},\, k\in K_{n}$. Thus, there exists a Schwartz
function $\phi''\in\cS(\CC^{n})$, such that
\[
|W(tk)W'(tk)\Phi((0,0,...,0,1)g)|\le\delta(t)|\det t|_{\CC}^{\lambda}|t_{n}|_{\CC}^{-n\lambda}\phi''\left(\frac{t_{1}}{t_{2}},...,\frac{t_{n-1}}{t_{n}},t_{n}\right)
\]
for all $k\in K_{n},t\in T_{n}$. Let us rewrite the expression
\[
\int\limits _{U_{n}(\CC)\backslash G_{n}(\CC)}|W(g)W'(g)\Phi((0,0,...,0,1)g)||\det g|_{\CC}^{Re(s)}dg
\]
using the Iwasawa decomposition. We obtain
\begin{equation}
\int\limits _{K_{n}}\int\limits _{T_{n}}|W(tk)W'(tk)\Phi((0,0,...,0,1)g)||\det t|_{\CC}^{Re(s)}\delta^{-1}(t)dtdk.\label{eq:Iwasawa complex}
\end{equation}
Indeed, for $f:G_{n}(\CC)\to\CC$ such that the following integrals
are absolutely convergent, we have
\[
\begin{split}\int_{G_{n}(\CC)}f(g)dg & =\int\limits _{U_{n}(\CC)}\int\limits _{T_{n}(\CC)}\int\limits _{K_{n}}f(tuk)dudtdk=\int\limits _{U_{n}(\CC)}\int\limits _{T_{n}(\CC)}\int\limits _{K_{n}}f((tut^{-1})tk)dudtdk\\
 & = \int\limits _{U_{n}(\CC)}\int\limits _{T_{n}(\CC)}\int\limits _{K_{n}}f(utk)\delta^{-1}(t)dudtdk.
\end{split}
\]
The integrand in (\ref{eq:Iwasawa complex}) is bounded by
\[
\begin{split}& |W(tk)W'(tk)\Phi((0,0,...,0,1)g)||\det t|_{\CC}^{Re(s)}\delta^{-1}(t)\le  \phi''\left(\frac{t_{1}}{t_{2}},...,\frac{t_{n-1}}{t_{n}},t_{n}\right)|\det t|_{\CC}^{Re(s)+\lambda}|t_{n}|_{\CC}^{-n\lambda}\\
& \times \left|\frac{t_{1}}{t_{2}}\right|_{\CC}^{Re(s)+\lambda}
\left|\frac{t_{2}}{t_{3}}\right|_{\CC}^{2(Re(s)+\lambda)}...\cdot\left|\frac{t_{n-1}}{t_{n}}\right|_{\CC}^{(n-1)(Re(s)+\lambda)}\cdot|t_{n}|_{\CC}^{nRe(s)}.
\end{split}
\]
It follows that the integral absolutely converges for $s$ satisfying
$Re(s)>-\lambda$ and $Re(s)>0$. As $\lambda>-\frac{1}{2}$, we obtain
the absolute convergence of the Rankin-Selberg integral at $s=\frac{1}{2}$. \end{proof}
\end{section} 
%
\begin{section}{Integral representation of Whittaker functions}
Let $n\ge2$ be fixed and let $K=U_{n}(\CC)$ be a maximal compact subgroup
of $G_{n}(\CC)$. Let $(\pi,V)$ be an irreducible, $G_{n}(\RR)$-distinguished
representation of $G_{n}(\CC)$. Suppose $(\pi,V)$ is unitarizable
and generic and let $\cW(\pi,\psi)$ be its Whittaker model. The functional
\[
\mu:W\mapsto\int\limits _{U_{n-1}(\RR)\backslash G_{n-1}(\RR)}W(h)dh
\]
defines an $P_{n}(\RR)$-invariant functional on $\cW(\pi,\psi)$,
see \cite[Lemma 1.2]{LapMao} for the proof of
the absolute convergence of the functional $\mu$. We will identify
in the sequel the functional $\mu$ on $\cW(\pi,\psi)$ with the corresponding
linear functional on $V$, which we will also denote by abuse of notation
$\mu$. By the uniqueness of Whittaker model, this identification
defines $\mu\in V^{*}$ in a unique way, up to a scalar multiple.
Since $\mu\in(V^{*})^{P_{n}(\RR)}$ and $(V^{*})^{P_{n}(\RR)}=(V^{*})^{G_{n}(\RR)}$,
see \cite[Theorem 1.1]{Kem}, we obtain that $\mu\in(V^{*})^{G_{n}(\RR)}.$
Clearly, $\mu\ne0$. \\
The functional $\mu$ defines an embedding of $V$ to functions on
$G_{n}(\RR)\backslash G_{n}(\CC)$ via
\[
V\ni v\mapsto\left(g\mapsto\mu(\pi(g)v)\right).
\]
By abuse of notation we denote this embedding again by $\mu$. Denote
the image of embedding $\mu$ by $\cC_{G_{n}(\RR)}(\pi)$. In the
other direction, we can define a map $$\theta:\cC_{G_{n}(\RR)}(\pi)\to\cW(\pi,\psi)$$
by
\begin{equation}
\theta:f\mapsto\left(g\mapsto\int\limits _{U_{n}(\RR)\backslash U_{n}(\CC)}f(ug)\psi^{-1}(u)du\right).\label{eq:integral Whit}
\end{equation}
In this section we will prove that for every $n$ there exists an
irreducible representation $(\pi,V)$ of $G_{n}(\CC)$ that is $G_{n}(\RR)$-distinguished
and such that the integral (\ref{eq:integral Whit}) is absolutely
convergent for every $K$-finite vector in $(\pi,V)$.\\
Suppose for the moment that we have a generic distinguished irreducible
representation $(\pi,V)$ of $G_{n}(\CC)$ and the integral (\ref{eq:integral Whit})
is absolutely convergent for every $f\in\cC_{G_{n}(\RR)}(\pi)$. Then,
from \cite{LapMao}, the composition of maps $\theta\circ\mu=c\mu$
for some constant $0\ne c\in\CC$. Hence, $\theta\ne0$, and since
$(\pi,V)$ is irreducible, we get that $\mu$ is an isomorphism. We
state a consequence
\begin{lem}
\label{lem:integral representation} Let $(\pi,V)$ be an irreducible,
distinguished, $G_{n}(\RR)$-generic representation of $G_{n}(\CC)$.
Suppose the integral
$\int\limits _{U_{n}(\RR)\backslash U_{n}(\CC)}W(u)\psi^{-1}(u)du$
absolutely converges for every $W\in\cW(\pi,\psi)$. Then for every
$W\in\cW(\pi,\psi)$ there exists $f\in\cC_{G_{n}(\RR)}(\pi)$ such
that
\[
W(g)=\int\limits _{U_{n}(\RR)\backslash U_{n}(\CC)}f(ug)\psi^{-1}(u)du.
\]

\end{lem}
Recall the decomposition (\ref{eq: KAH}) $G_{n}(\CC)=G_{n}(\RR)A^{+}K$.
The involution $g\to^{t}\hspace{-1mm}g^{-1}$ preserves this decomposition.
Let $\check{W}(g)=W(^{t}g^{-1})$. The Whittaker model $\cW(\tilde{\pi},\psi^{-1})$
of the contragredient representation of $(\pi,V)$ is given by
\[
\cW(\tilde{\pi},\psi^{-1})=\{\check{W}:W\in\cW(\pi,\psi)\}.
\]
If conditions of Lemma \ref{lem:integral representation} are satisfied
for $\cW(\pi,\psi)$ then they are satisfied also for the contragredient
representation $\cW(\tilde{\pi},\psi^{-1})$. Explicitly, if $W\in\cW(\pi,\psi)$
equals to
\[
W(g)=\int\limits _{U_{n}(\RR)\backslash U_{n}(\CC)}f(ug)\psi^{-1}(u)du,
\]
then
\[
\check{W}(g)=\int\limits _{U_{n}(\RR)\backslash U_{n}(\CC)}\check{f}(ug)\psi(u)du,
\]
where $\check{f}(g):=f(^{t}g^{-1})$. For $g\in G_{n}(\CC)$ let
\[
||g||=\sum_{i,j=1}^{n}|g_{ij}|^{2}.
\]
Let $||g||_{H}=\max\{||g||,||g^{-1}||\}$. Then $||g||_{H}$ is a
norm on $G$ in the sense of \cite[section 2.A.2]{WallachB1}.
\begin{lem}
\label{lem:fast decay}Let $N>0$. Then there exists an irreducible
$G_{n}(\RR)$-distinguished representation $(\pi,V)$ of $G_{n}(\CC)$
such that for every $K$-finite function $f\in V$ there exists a
constant $C>0$, depending only on $f$, such that for every $k\in K,a\in A,h\in G_{n}(\RR)$
there is an inequality
\begin{equation}
|f(hak)|\le C(f)||a||_{H}^{-N}.\label{ineq: fast decay}
\end{equation}
\end{lem}
\begin{proof}
By \cite{Flensted-Jensen}, there exists relatively
discrete series $\mathcal{H}:=L_{ds}^{2}(G_{n}(\RR)\backslash G_{n}(\CC))$.
Moreover, every irreducible representation in $\cH$ is isomorphic
to some $I(\chi)$, where
\[
\chi(z)=\left((z/|z|)^{i_{1}},(z/|z|)^{i_{2}},...,(z/|z|)^{i_{n}}\right)
\]
and $i_{1},...,i_{n}\in\ZZ$. If $C>0$ us bug enough and $|i_{k}-i_{j}|>C>0$
for all $i\ne j$, then the $G_{n}(\RR)\backslash G_{n}(\CC)$ model
of the space $I(\chi)$ lies in $\cH$, and the $(\gotG,K)$-module
generated by a $K$-finite function $0\ne f_{\lambda}\in I(\chi)$
satisfies the properties of the Lemma. Indeed, by \cite[p.
254]{Flensted-Jensen}, see also \cite[Proposition
5.1]{Kassel Kobayashi}, if $C>0$ is big enough and $|i_{k}-i_{j}|>C>0$
for all $j\ne k$, then $f_{\lambda}(hak)\le C'||a||^{-N}$ for all
$h\in H,a\in A^{+},$ and $k\in K$.\\
Clearly, $f_{\lambda}$ and right translations of $f_{\lambda}$ by
$K$ satisfy the properties of our Lemma. We should prove that the
derivatives of $f_{\lambda}$ also satisfy similar growth properties.
This is achieved by a classical idea, which is attributed to Harish-Chandra,
see also an expository article by \cite{Cowling-Haag-H}.
The function $f_{\lambda}$ is $K$-finite, hence, there exists a
smooth function $e_{\alpha}$ of a compact support such that $f_{\lambda}*e_{\alpha}=f_{\lambda}$.
Thus, for $X\in\gotG$ we have $dX(f_{\lambda})=f_{\lambda}*dX(e_{\alpha})$.
It follows that the derivatives of $f_{\lambda}$ have the same decay
properties that $f_{\lambda}$ has.\\
Finally, the $(\gotG,K)$-module generated by $f_{\lambda}$ is of
finite length. Consequently, it contains an irreducible admissible
$(\gotG,K)$-submodule with satisfy the decay property (\ref{ineq: fast decay}).
\end{proof}
If every $K$-finite function in $(\pi,V)$ satisfies (\ref{ineq: fast decay})
we say that the representation $(\pi,V)$ is of decay faster than
$N$. Note that if $(\pi,V)$ is of decay faster than $N$, then its
contragredient $(\tilde{\pi},\tilde{V})$ is also of decay faster
than $N$. Indeed, we can realize $(\tilde{\pi},\tilde{V})$ as $\tilde{V}=\{\check{f}:f\in V\}$,
where $\check{f}(g):=f(^{t}\hspace{-1mm}g^{-1}).$ If $g=hak$, then
$^{t}g^{-1}=^{t}\hspace{-1mm}h^{-1}a\hspace{1mm}^{t}k^{-1}$. Hence,
the property of fast decay is true for $\check{f}$ if and only if
it is true for $f$.\\
To obtain estimates of convergence of integrals over the unipotent
matrices we need the next elementary result. Define $\Omega_{n}$
as the subset of all upper triangular unipotent matrices in $G_{n}(\CC)$
with with $u_{ij}$ purely imaginary for $j>i$. Note that $\Omega_{n}$
is a fundamental domain for $U_{n}(\RR)\backslash U_{n}(\CC).$
\begin{lem}
\label{lem: unipotents} There exist $C>0$ and $d>0$, which depend
only on $n$, such that for every $u\in\Omega_{n}$ we have
\[
||u||\le C||u\bar{u}^{-1}||^{d}.
\]
 \end{lem}
\begin{proof}
The proof is by induction on $n$. For $n=2$ it follows from a direct
computation: let $u=\left(\begin{array}{cc}
1 & ix\\
0 & 1
\end{array}\right)$. Then $u\bar{u}^{-1}=\left(\begin{array}{cc}
1 & 2ix\\
0 & 1
\end{array}\right)$, and the claim is satisfied. \\
For a general $n$, we can partition the set of the entries appearing
in $g$ by the value of $j-i$: that is,
\[
A_{0}=\{g_{11},g_{22},...,g_{nn}\},\, A_{1}=\{g_{12},g_{23},...,g_{(n-1)n}\},...,\, A_{n-1}=\{g_{1n}\}.
\]
Denote by $B_{j}:=\cup_{0\le i<j}A_{i}$. The crucial observation
is that entry $(i,j)$ of $\bar{u}^{-1}$ with indexes satisfying
$j-i=k$ equals to
\[
\bar{u}_{ij}^{-1}=u_{ij}+P_{ij}(u),
\]
where $P_{ij}\in\CC[B_{k}]$ is some fixed polynomial which depends
only on the entries $u_{lm}$ with indexes $l-m<k$. Similarly,
\[
(u\bar{u}^{-1})_{ij}=2u_{ij}+Q_{ij}(u),
\]
where $Q_{ij}\in\CC[B_{k}]$ is some fixed polynomial which depends
only on the entries $u_{lm}$ with indexes $l-m<k$. For example,
let $n=3$,
\[
u=\left(\begin{array}{ccc}
1 & ix & iy\\
0 & 1 & iz\\
0 & 0 & 1
\end{array}\right),\bar{u}^{-1}=\left(\begin{array}{ccc}
1 & ix & iy-xz\\
0 & 1 & iz\\
0 & 0 & 1
\end{array}\right),u\bar{u}^{-1}=\left(\begin{array}{ccc}
1 & ix & 2iy-2xz\\
0 & 1 & iz\\
0 & 0 & 1
\end{array}\right).
\]
Thus $P_{12}=P_{23}=0,P_{13}=-xz,Q_{12}=Q_{23}=0,Q_{13}=-2xz$. Denote
$v=u\bar{u}^{-1}$. Define "partial semi-norms"
of $u$ by
\[
||u||_{k}=\sum_{i,j:j-i\le k}|u_{ij}|^{2}.
\]
We will prove by induction on $k$ (base is $k=1$) that for every
$k$, there exist $C_{k},d_{k}>0$ such that $||u||_{k}\le C_{k}||v||_{k}^{d_{k}}$.
As $||u||_{n}=||u||$, the result follows.\\
Base: $k=1$. Simple calculation shows that for $C_{1}=1,d_{1}=1$
we obtain the desired inequality. Suppose the claim is true for $k-1$,
that is,
\[
||u||_{k-1}\le C_{k-1}||v||_{k-1}^{d_{k-1}}.
\]
We want to show a similar inequality for $k$. There exists $C,d>0$
such that for every $1\le i\le n-k$ we have $|v_{i,i+k}|\ge|u_{i,i+k}|-C||u||_{k-1}^{d}$.
For example, one can choose
\[
d=\max\{deg(P_{ij}):i-j=k\},
\]
and big enough constant $C$. Let $u$ be a given upper triangular
unipotent matrix with purely imaginary up entries above the diagonal.
There exist constants $C',C''$ such that if for all $i$ we have
$|u_{i,i+k}|\le2C||u||_{k-1}^{d}$, then
\[
||u||_{k}=||u||_{k-1}+\sum_{i}|u_{i,i+k}|^{2}\le C'||u||_{k-1}^{2d}\le C''||v||_{k-1}^{2dd_{k-1}}\le C''||v||_{k}^{2dd_{k-1}}.
\]
On the other hand, if for some $i$ we have $|u_{i,i+k}|>2C||u||_{k-1}^{d}$,
then we have an inequality $|v_{i,i+k}|>\frac{|u_{i,i+k}|}{2}$ and
there exist a constant $C'''$ such that
\[
\sum_{i=1}^{n-k}|u_{i,i+k}|^{2}\le C'''\sum_{i=1}^{n-k}|v_{i,i+k}|^{2}\le C'''||v||_{k}\le C'''||v||_{k}^{2dd_{k-1}}.
\]
Therefore, in both cases there are constants $C_{k},d_{k}$ such that
\[
||u||_{k}\le C_{k}||v||_{k}^{d_{k}}.
\]
\end{proof}
\begin{cor}
There exist $C>0$ and $d>0$, which depend only on $n$, such that
for every $u\in\Omega_{n}$ we have
\[
||u||_{H}\le C||u\bar{u}^{-1}||_{H}^{d}.
\]
\end{cor}
\begin{proof}
From Lemma \ref{lem: unipotents} we know that there exist $C_{1}>0,d_{1}>0$
such that for every $u\in\Omega_{n}$ we have $||u||<C_{1}||u\bar{u}^{-1}||^{d_{1}}$.
Similarly, one proves that there exist $C_{2}>0,d_{2}>0$ such that
for every $u\in\Omega_{n}$ we have $||u||<C_{2}||\bar{u}u^{-1}||^{d_{2}}$.
Define $C=\max\{C_{1},C_{2}\},d=\max\{d_{1},d_{2}\}$. Then $||u||_{H}\le C||u\bar{u}^{-1}||_{H}^{d}$
for every $u\in\Omega_{n}$. \end{proof}
\begin{lem}
Let $N>0$ be big enough. Then for every irreducible,$G_{n}(\RR)$-distinguished
representation $(\pi,V)$ of $G_{n}(\CC)$ with decay faster than
$N$, the integral
\[
\int\limits _{U_{n}(\RR)\backslash U_{n}(\CC)}f(ug)du
\]
absolutely converges for every $g\in G_{n}(\CC)$ and every $K$-finite
function $f\in V$.\end{lem}
\begin{proof}
Let $(\pi,V)$ be a $G_{n}(\RR)$-distinguished representation of
$G_{n}(\CC)$ such that for every $K$-finite function $f\in V$
there exists $C>0$ depending only on $f$ such that $$|f(hak)|<C||a||_{H}^{-N}$$
for every $h\in G_{n}(\RR),a\in A^{+},k\in K$. Write $ug=hak$. Then
$(\overline{ug})^{-1}ug=\bar{g}^{-1}(\bar{u}^{-1}u)g.$ Since $g$
is fixed, there exists $C_{1}>0$ such that for every matrix $u\in G_{n}$
we have
\[
C_{1}^{-1}||\bar{u}^{-1}u||<||(\overline{ug})^{-1}ug||<C_{1}||\bar{u}^{-1}u||.
\]
By Lemma \ref{lem: unipotents}, for $u\in\Omega$ we have
\[
||u||<C_{2}||\bar{u}^{-1}u||^{d}.
\]
On the other hand ,
\[
\left(\overline{ug}\right)^{-1}ug=\bar{k}^{-1}(a\bar{a}^{-1})k=\bar{k}^{-1}a^{2}k.
\]
Note that $k\in K$ is a unitary matrix, therefore
\[
||\bar{k}^{-1}a^{2}k||=||a^{2}k||=||a^{2}||.
\]
Combining these inequalities we get
\[
||a^{2}||=||(\overline{ug})^{-1}ug||>C_{3}||u\bar{u}^{-1}||>C_{4}||u||^{1/d}.
\]
Finally, we obtain that there exist constants $C,d'$ such that for
$ug=hak$, where $u\in\Omega_{n}$ and $g\in G_{n}$ is fixed, we
have
\[
||a||>C||u||^{1/d'}
\]
Therefore,
\begin{equation}
\int\limits _{U_{n}(\RR)\backslash U_{n}(\CC)}|f(ug)|du\le\int\limits _{\Omega}C||u||_{H}^{-N/d'}du. \label{eq: convergence estimates}
\end{equation}
The integral in (\ref{eq: convergence estimates}) converges for $N$
big enough, thus the lemma is proved. \end{proof}
\begin{cor}
Let $N>0$ be big enough. Then for every irreducible $G_{n}(\RR)$-distinguished
representation $(\pi,V)$ of $G_{n}(\CC)$ with decay faster than
$N$, the integral (\ref{eq:integral Whit}) is absolutely convergent. \end{cor}
\end{section} 
%
\begin{section}{\label{sec:Asai}Archimedean Asai integrals}

In \cite{Flicker1} Flicker introduced non-archimedean
Asai integrals. He used them in \cite{Flicker2} to
analyze the local and global Asai $L$ and $\epsilon$-factors. In
this section we introduce an archimedean analog of Asai integrals
and prove that they are of moderate growth. Also, we state a functional
equation they satisfy analogous to \cite[Lemma 4.2]{O}.\\
Let $(\pi,V)$ be a generic irreducible unitarizable representation
of $G_{n}(\CC)$ and let $\cW(\pi,\psi)$ be its Whittaker model.
For $W\in\cW(\pi,\psi)$, we define an archimedean Asai integral to
be
\begin{equation}
Z(s,W,\Phi)=\int\limits _{U_{n}(\RR)\backslash G_{n}(\RR)}W(g)\Phi((0,0,...,0,1)g)|\det g|_{\RR}^{s}dg.\label{eq:Asai}
\end{equation}

\begin{lem}
\label{lem:Lapid integral} The integral (\ref{eq:Asai}) is absolutely
convergent for $Re(s)\ge1$. Moreover, there exists $N>0$ such that
for every $g\in G$, every $K$-finite function $W'\in\mathcal{W}(\pi,\psi),$
and every $\Phi\in\mathcal{S}(\CC^{n})$ we have
\[
\int\limits _{U_{n}(\RR)\backslash G_{n}(\RR)}|W'(hg)\Phi((0,0,...,0,1)hg)||\det h|_{\RR}dh\le C(W',\Phi)||g||^{N}.
\]
\end{lem}
\begin{proof}
Denote by $K_{n}(\RR)$ the standard maximal compact subgroup of $G_{n}(\RR)$
and by $T_{n}(\RR)$ the diagonal torus in $G_{n}(\RR)$. Let $\tilde{\delta}$
be the modulus function of $G_{n}(\RR)$. Using Iwasawa decomposition
we obtain
\begin{equation}
|Z(s,W,\Phi)|\le\int\limits _{K_{n}(\RR)}\int\limits _{T_{n}(\RR)}|W(tk)\Phi((0,0,...,0,1)tk)||\det(t)|_{\RR}^{Re(s)}\tilde{\delta}^{-1}(t)dt\, dk.\label{eq:Whittaker est}
\end{equation}
By \cite[Corollary 2.2]{LapMao}, there exists
$\phi\in\cS(\RR^{n-1})$ such that
\[
|W(tk)|\le\delta^{\frac{1}{2}}(t)|\det t|_{\RR}^{2\lambda}|t_{n}|^{-2n\lambda}\phi\left(\frac{t_{1}}{t_{2}},\frac{t_{2}}{t_{3}},...,\frac{t_{n-1}}{t_{n}}\right)
\]
for all $t\in T_{n}(\RR),\, k\in K_{n}(\RR)$. Note that $\delta^{\frac{1}{2}}(t)\tilde{\delta}(t)^{-1}=1$,
and there exists a function $\phi'\in\cS(\RR^{n})$ such that
\[
\phi\left(\frac{t_{1}}{t_{2}},\frac{t_{2}}{t_{3}},...,\frac{t_{n-1}}{t_{n}}\right)\Phi((0,0,...,0,1)tk)\le\phi'\left(\frac{t_{1}}{t_{2}},\frac{t_{2}}{t_{3}},...,\frac{t_{n-1}}{t_{n}},t_{n}\right)
\]
 for all $t\in T_{n}(\RR),\, k\in K_{n}(\RR)$. Hence, the integrand in (\ref{eq:Whittaker est})
is bounded by
\[
\begin{split}& |W(tk)\Phi((0,0,...,0,1)tk)||\det(t)|_{\RR}^{Re(s)}\tilde{\delta}^{-1}(t)\le |\det t|^{2\lambda+Re(s)}|t_{n}|^{-2n\lambda}\phi'\left(\frac{t_{1}}{t_{2}},\frac{t_{2}}{t_{3}},...,\frac{t_{n-1}}{t_{n}},t_{n}\right)\\
& \le  \left|\frac{t_{1}}{t_{2}}\right|^{2\lambda+Re(s)}\left|\frac{t_{2}}{t_{3}}\right|^{2(2\lambda+Re(s))}\cdot...\cdot\left|\frac{t_{n-1}}{t_{n}}\right|^{(n-1)(2\lambda+Re(s))}|t_{n}|^{nRe(s)}\phi'\left(\frac{t_{1}}{t_{2}},\frac{t_{2}}{t_{3}},...,\frac{t_{n-1}}{t_{n}},t_{n}\right).
\end{split}
\]
Consequently, the integral (\ref{eq:Whittaker est}) converges absolutely
for $s$ satisfying $Re(s)+2\lambda>0$ and $Re(s)>0$. Note that
for $\pi$ unitary and generic one have $\lambda>-\frac{1}{2}$, see
\cite[page 8]{LapMao}. Thus, the integral (\ref{eq:Whittaker est}) converges
absolutely for $Re(s)\ge1$.\\
From the proof of the absolute convergence of $Z(1,W,\Phi)$, it follows
also that there exists $N>0$, depending only on the representation
$\pi$ such that for every $W'\in\cW(\pi,\psi)$ and $\Phi\in\cS(\RR^{n})$
there exists $C(W',\Phi)$ such that
\[
|Z(1,\pi(g)W',R(g)\Phi)|\le C(W',\Phi)||g||^{N}
\]
for all $g\in G_{n}(\CC)$.
\end{proof}
Next lemma provides a functional equation for archimedean Asai integrals.
\begin{lem}
Let $\pi$ be an irreducible, unitary , non-degenerate and $G_{n}(\RR)$-distinguished
representation of $G_{n}(\CC)$. For every $\Phi\in\cS(\CC^{n})$
and $W\in\cW(\pi,\psi)$ we have
\[
Z(1,\tilde{W},\hat{\Phi}|_{\RR^{n}})=c(\pi)Z(1,W,\Phi|_{\RR^{n}}).
\]
 \end{lem}
\begin{proof}
For the proof see \cite[Lemma 4.2]{O}.
\end{proof}
We will use the following technical result in the next section.
\begin{lem}
\label{lem:Whittaker + Asai} Let $(\pi',V')$ be non-degenerate unitary
representation of $G_{n}(\CC)$ and let $\mathcal{W}(\pi',\psi^{-1})$
be its Whittaker model. Then there exists $N>0$ such that for every
irreducible, $G_{n}(\RR)$-distinguished representation $(\pi,V)$
of $G_{n}(\CC)$ with decay faster than $N$ and every function $f\in\cC_{G_{n}(\RR)}(\pi)$,
the following integral is absolutely convergent:
\[
\int\limits _{G_{n}(\RR)\backslash G_{n}(\CC)}|f(g)||\det g|_{\CC}^{\frac{1}{2}}\left(\int\limits _{U_{n}(\RR)\backslash G_{n}(\RR)}|W'(hg)\Phi((0,0,...,0,1)hg)||\det h|_{\RR}dh\right)dg.
\]
\end{lem}
\begin{proof}
This is an immediate consequence of Lemma \ref{lem:Lapid integral}.
\end{proof}
\end{section} 
%
\begin{section}{\label{sec: functionals}Equality of two functionals}

Let $(\pi,V)$ be an irreducible, non-degenerate, and unitarizable
$G_{n}(\RR)$-distinguished representation of $G_{n}(\CC)$ and let
$\mathcal{W}(\pi,\psi)$ be its Whittaker model. Define linear functionals
$\mu,\tilde{\mu}\in V^{*}$ on $\mathcal{W}(\pi,\psi))$ by
\[
\mu:W\mapsto\int\limits _{U_{n-1}(\RR)\backslash G_{n-1}(\RR)}W(g)dg,
\]
and
\[
\tilde{\mu}:W\mapsto\int\limits _{U_{n-1}(\RR)\backslash G_{n-1}(\RR)}W\left(\left(\begin{array}{cc}
0 & 1\\
I_{n-1} & 0
\end{array}\right)\left(\begin{array}{cc}
g & 0\\
0 & 1
\end{array}\right)\right)dg.
\]
Since $\mu,\tilde{\mu}\in(V^{*})^{P_{n}(\RR)}$
and $(V^{*})^{P_{n}(\RR)}=(V^{*})^{G_{n}(\RR)}$, see \cite[Theorem
1.1]{Kem}, we obtain that $\mu,\tilde{\mu}\in(V^{*})^{G_{n}(\RR)}.$
Clearly, the functionals $\mu,\tilde{\mu}$ are nonzero. The space of $G_{n}(\RR)$-invariant continuous functionals on
$V$ is one-dimensional, see \cite[Theorem 8.2.5]{AG2},
thus there exists a proportionality constant $c(\pi)\ne0$ such that
$\tilde{\mu}=c(\pi)\mu$.\\
The goal of this section is to calculate the proportionality factor
$c(\pi)$ by proving the following theorem
\begin{thm}
Let $\pi$ be an irreducible, $G_{n}(\RR)$-distinguished representation
of $G_{n}(\CC)$. Then $c(\pi)=1$.
\end{thm}
We state now a lemma, which is an archimedean analogue of \cite[Lemma
6.1]{O}.
\begin{lem}
\label{lem:RS} Let $\pi'$ be an irreducible, unitarizable, generic
and $G_{n}(\RR)$-distinguished representation of $G_{n}(\CC)$. Then
there exists an irreducible, unitarizable, generic and $G_{n}(\RR)$-distinguished
representation $\pi$ of $G_{n}(\CC)$ such that
\[
\gamma(\frac{1}{2},\pi\times\pi';\psi)=c(\pi').
\]

\end{lem}
Note that we already know that for $\pi,\pi'$ as in Lemma \ref{lem:RS}
we have $\gamma(\frac{1}{2},\pi\times\pi',\psi)=1$. As a result,
the equality $c(\pi')=1$ follows.\\
The proof of Lemma \ref{lem:RS} is same to the proof of \cite[Lemma
6.1]{O}. However, in the archimedean case, there are convergence
issues thatwe need to check.
\begin{proof}[Proof of Lemma \ref{lem:RS}] Let $W\in\mathcal{W}(\pi,\psi),\, W'\in\mathcal{W}(\pi',\psi^{-1})$
and $\Phi\in\mathcal{S}(\CC^{n})$. The idea is to prove an equality
of Rankin-Selberg integrals of the type
\begin{equation}
\Psi(\frac{1}{2},\tilde{W},\tilde{W'};\hat{\Phi})=c(\pi')\Psi(\frac{1}{2},W,W';\Phi).\label{eq:Rankin-Selberg ints}
\end{equation}
Actually, it is enough to prove such an equality for a pair of functions
$W\in\mathcal{W}(\pi,\psi)$ and $W'\in\mathcal{W}(\pi',\psi)$ such
that at least one of the integrals $\Psi(\frac{1}{2},\tilde{W},\tilde{W'};\hat{\Phi}),\Psi(\frac{1}{2},W,W';\Phi)$
is nonzero ( and thus both integrals are nonzero).\\
We will obtain the necessary convergence estimates for every $K_{n}$-finite
function $W\in\mathcal{W}(\pi,\psi)$ and every function $W'\in\mathcal{W}(\pi',\psi)$.
By our classification of $G_{n}(\RR)$-distinguished representations
of $G_{n}(\CC)$, the central character $\omega_{\pi}$ of $G_{n}(\RR)$-distinguished
representation satisfies $\omega_{\pi}(-1)=1$. Thus, by Theorem \ref{thm: Jacquet func_eqn}, we have an equality
\[
\Psi(1-s,\tilde{W},\tilde{W'};\hat{\Phi})=\gamma(s,\pi\times\pi',\psi)\Psi(s,W,W';\Phi).
\]
Let $f\in\mathcal{C}_{G_{n}(\RR)}(\pi)$ be such that
\[
W(g)=\int\limits _{U_{n}(\RR)\backslash U_{n}(\CC)}f(ug)\psi^{-1}(u)du.
\]
We will prove the absolute convergence of the following integrals
at $s=\frac{1}{2}$:
\begin{equation}
\begin{split}&\int\limits _{U_{n}(\CC)\backslash G_{n}(\CC)} |W(g)W'(g)\Phi((0,0,...,0,1)g)||\det g|_{\CC}^{s}dg\le\\
&\int\limits _{U_{n}(\CC)\backslash G_{n}(\CC)}\left( \int\limits _{U_{n}(\RR)\backslash U_{n}(\CC)}|f(ug)|du\right) |W'(g)\Phi((0,0,...,0,1)g)||\det g|_{\CC}^{s}dg\\
&=\int\limits _{U_{n}(\CC)\backslash G_{n}(\CC)}|f(g)W'(g)\Phi((0,0,...,0,1)g)||\det g|_{\CC}^{s}dg\\
&=\int\limits _{G_{n}(\RR)\backslash G_{n}(\CC)}|f(g)||\det g|_{\CC}^{s}\left( \int\limits _{U_{n}(\RR)\backslash G_{n}(\RR)}|W'(hg)\Phi((0,0,...,0,1)hg)||\det h|_{\RR}^{2s}dh \right)dg.
\end{split}
\label{eq: convergence}
\end{equation}
Indeed, the left-hand side of (\ref{eq: convergence}) is absolutely
convergent by Lemma \ref{lem: RS convergence} and the integrals in
the right-hand side of (\ref{eq: convergence}) are absolutely convergent
by Lemmas \ref{lem:Lapid integral} and \ref{lem:Whittaker + Asai}.
Using absolute convergence for $s=\frac{1}{2}$ of the integrals appearing
in (\ref{eq: convergence}) we obtain the equality (\ref{eq:Rankin-Selberg ints})by
the following argument. We have
\[
\begin{split}
\Psi(\frac{1}{2},W,W';\Phi) & =\int\limits _{U_{n}(\CC)\backslash G_{n}(\CC)}W(g)W'(g)\Phi((0,0,...,0,1)g)|\det g|_{\CC}^{\frac{1}{2}}dg\\
 & =\int\limits _{U_{n}(\CC)\backslash G_{n}(\CC)}\left(\int\limits _{U_{n}(\RR)\backslash U_{n}(\CC)}f(ug)\psi_{n}^{-1}(u)du\right)W'(g)\Phi((0,0,...,0,1)g)||\det g|_{\CC}^{\frac{1}{2}}dg\\
 & =\int\limits _{U_{n}(\CC)\backslash G_{n}(\CC)}f(g)W'(g)\Phi((0,0,...,0,1)g)|\det g|_{\CC}^{\frac{1}{2}}dg\\
 & =\int\limits _{G_{n}(\RR)\backslash G_{n}(\CC)}f(g)|\det g|_{\CC}^{\frac{1}{2}}\left(\int\limits _{U_{n}(\RR)\backslash G_{n}(\RR)}W'(hg)\Phi((0,0,...,0,1)hg)|\det h|_{\RR}dh\right)dg\\
 & \int\limits _{G_{n}(\RR)\backslash G_{n}(\CC)}f(g)|\det g|_{\CC}^{\frac{1}{2}}Z(1,\pi'(g)W',\Phi(\cdot g)|_{\RR^{n}})dg.
\end{split}
\]
Define $f^{*}(g):=f(^{t}g^{-1})$. Clearly, $f^{*}\in\cC_{G_{n}(\RR)}(\tilde{\pi}).$
Applying the change of variables $u\to w_{n}\hspace{1mm}^{t}u^{-1}w_{n}^{-1}$,
and the fact that$f(w_{n}g)=f(g)$, it follows from the definitions
that
\[
\tilde{W}(g)=\int\limits _{U_{n}(\RR)\backslash U_{n}(\CC)}f^{*}(ug)\psi(u)du.
\]
The same computation applied to $\tilde{\pi}$ and $\tilde{\pi}'$
yields
\[
\Psi(\frac{1}{2},\tilde{W},\tilde{W}';\hat{\Phi})=\int\limits _{G_{n}(\RR)\backslash G_{n}(\CC)}f^{*}(g)|\det g|_{\CC}^{\frac{1}{2}}Z(1,\tilde{\pi}'(g)\tilde{W}',\hat{\Phi}(\cdot g)|_{\RR^{n}})dg.
\]
It follows that
\[
Z(1,\tilde{\pi}'(^{t}\hspace{-1mm}g^{-1})\tilde{W}',\hat{\Phi}(\cdot^{t}\hspace{-1mm}g^{-1})|_{\RR^{n}})=c(\pi')|\det g|_{\CC}Z(1,\pi'(g)W',\Phi(\cdot g)|_{\RR^{n}}).
\]
Finally, we obtain
\[
\Psi(\frac{1}{2},\tilde{W},\tilde{W}';\hat{\Phi})=c(\pi')\Psi(\frac{1}{2},W,W';\Phi)
\]
for every $K_{n}$-finite functions $W\in\cW(\pi,\psi)$ , $W'\in\cW(\pi',\psi^{-1})$
and every $\Phi\in\cS(\CC^{n})$. It is well-known that there exists
$K_{n}$-finite $W\in\cW(\pi,\psi),\, W'\in\cW(\pi',\psi^{-1})$ such
that $\Psi(\frac{1}{2},W,W';\Phi)\ne0$. It follows that $c(\pi')=\gamma(\frac{1}{2},\pi\times\pi';\psi)$. \end{proof}
\end{section} 
%
\appendix
\section{Generic Langlands quotient}
In this section we sketch a proof of the well-known fact that the Langalnds quotient of
$I(\chi)$ is generic if and only if $I(\chi)$ is irreducible. This fact follows from the papers of Kostant \cite{Kostant} and
Vogan \cite{Vogan}. For the convenience of the reader we rewrite it
here.
Similar results for $GL_n(\RR)$ were obtained by Casselman and Zuckerman.  \\
Let $\gotG = M_n(\CC)$ be the Lie algebra of $G_n(\CC)$ and let $K$ be the standard maximal compact subgroup of $G_n(\CC)$.
\begin{defn*}
An irreducible $(\gotG,K)$-module $X$ is called large if its annihilator in the universal enveloping algebra $U(\gotG)$
is a minimal primitive ideal.
We will say that a smooth irreducible representation $(\pi,V)$ of $G_n(\CC)$ is large if the corresponding $(\gotG,K)$ module
consisting of $K$-finite vectors in $V$ is large as a $(\gotG,K)$-module.
\end{defn*}
Let $\chi = (\chi_1,\chi_2, ..., \chi_n)$ be a character of $B_n(\CC)$ and suppose $|\chi_j(t)| = |t|^{\lambda_j}$ with
$\lambda_1 \ge \lambda_2 \ge ... \ge \lambda_n$.
By  \cite[Theorem 6.2]{Vogan}, if $(\sigma,W)$ is an irreducible subrepresentation of $I(\chi)$ then $(\sigma,W)$ is large. Suppose $(\pi,V)$ is
the Langlands quotient of $I(\chi)$ and suppose $(\pi,V)$ is generic.
Then by Kostant theorem $(\pi,V)$ is large.
On the other hand, \cite[Corollary 6.7]{Vogan} states that there is a unique large composition factor in the composition series for
$I(\chi)$. We obtain $(\pi,V) \simeq (\sigma,W)$ and thus $ (\sigma,W) = I(\chi)$, that is, $I(\chi)$ is an irreducible representation.
%
\section{Gamma Factors: Converse direction}

In this section we make some observations on the problem of the following
type. Fix a smooth, irreducible, generic and admissible representation
$(\pi,V)$ of $G_{n}(\CC)$. Suppose we know that
\begin{equation}
\gamma\left(\frac{1}{2},\pi\times\pi',\psi\right)=1\label{eq: appendix}
\end{equation}
for every $m\le k$ and every smooth irreducible $G_{m}(\RR)$-distinguished
representation $(\pi',V')$ of $G_{m}(\CC)$. What should be the minimal
$k$ such that (\ref{eq: appendix}) implies that $(\pi,V)$ is $G_{n}(\RR)$-distinguished?
In this section we give an answer to this question in the case $(\pi,V)$
is a unitary representation.\\
In the following two theorems we prove that $k=1$ is enough. Theorem
\ref{thm: converse A} is a particular case of Theorem
\ref{thm:converse B}. Nevertheless we state it and
prove it since the proof of Theorem \ref{thm: converse A} is simpler then the proof of Theorem \ref{thm:converse B} and in my opinion understanding it simplifies understanding
of the proof of Theorem \ref{thm:converse B}.
\begin{thm}
\label{thm: converse A} Let $\chi=(\chi_{1},\chi_{2},...,\chi_{n})$
be a unitary character of $B_{n}$ and suppose $\chi_{j}(z)=|z|_{\CC}^{s_{j}}(z/|z|)^{k_{j}}$
with $s_{j}$ purely imaginary and $k_{j}\in\ZZ$ for every $1\le j\le n$.
Suppose $(\pi,V)=I(\chi)$ is a smooth, generic and irreducible representation
of $G_{n}(\CC)$. Suppose
\[
\gamma\left(\frac{1}{2},\pi\times\chi',\psi\right)=1
\]
 for every $\RR^{\times}$-distinguished unitary character $\chi':\CC^{\times}\to\RR^{\times}$.
Then there exists an involution $w\in S_{n}$ such that $w\chi=\overline{(\chi^{-1})}$.
Moreover, one can find an involution $w\in S_{n}$ such that $w\chi=\overline{(\chi^{-1})}$
and also for every fixed point $w(i)=i$ the integer $k_{i}$ is even.\end{thm}
\begin{proof}
Observe that every $\RR^{\times}$-distinguished unitary character
$\chi':\CC\to\RR^{\times}$ is of the form $\chi(z)=\left(z/|z|\right)^{2m}$
for $m\in\ZZ$. By \cite[Lemma 16.3]{Jacquet}
we have
\[
\gamma\left(\frac{1}{2},Ind(\chi)\times\chi',\psi\right)=\prod\limits _{i=1}^{n}\gamma\left(\frac{1}{2},\chi_{i}\chi',\psi\right),
\]
 where $\gamma\left(\frac{1}{2},\chi_{i}\chi',\psi\right)$ is the
one-dimensional Tate's gamma factor. Following Tate denote $c_{m}(z)=(z/|z|)^{m}$
and recall that the Tate gamma factor is given by
\[
\gamma\left(s,c_{m},\psi\right)=\epsilon_{m}\frac{(2\pi)^{1-s}\Gamma\left(s+\frac{|m|}{2}\right)}{(2\pi)^{s}\Gamma\left((1-s)+\frac{|m|}{2}\right)},
\]
where
\[
\epsilon_{m}=\begin{cases}
1 & m\text{\, is even or }m>0\\
-1 & m\,\text{is odd and }m<0
\end{cases}.
\]
Let's rewrite the equality $\gamma\left(\frac{1}{2},Ind(\chi)\times\chi',\psi\right)=1$
as
\begin{equation}
\prod\limits _{i=1}^{n}\epsilon_{2m+k_{i}}\frac{(2\pi)^{\frac{1}{2}-s_{i}}}{(2\pi)^{\frac{1}{2}+s_{i}}}\frac{\Gamma\left(\frac{1}{2}+s_{i}+\frac{|k_{i}+2m|}{2}\right)}{\Gamma\left(\frac{1}{2}-s_{i}+\frac{|k_{i}+2m|}{2}\right)}=1\label{eq: prod_gamma}
\end{equation}
for every $m\in\ZZ$. The product in (\ref{eq: prod_gamma}) breaks
into 3 products:
\[
p_{m,1}=\prod\limits _{i=1}^{n}\epsilon_{2m+k_{i}},
\]

\[
p_{m,2}=\prod\limits _{i=1}^{n}\frac{(2\pi)^{\frac{1}{2}-s_{i}}}{(2\pi)^{\frac{1}{2}+s_{i}}}=(2\pi)^{-2s_{1}-2s_{2}-...-2s_{n}},
\]
and
\[
p_{m,3}=\prod\limits _{i=1}^{n}\frac{\Gamma\left(\frac{1}{2}+s_{i}+\frac{|k_{i}+2m|}{2}\right)}{\Gamma\left(\frac{1}{2}-s_{i}+\frac{|k_{i}+2m|}{2}\right)}.
\]
Note that the term $p_{m,2}$ is constant (does not depend on $m$)
and the term $p_{m,1}$ stabilizes, that is $p_{m,1}=p_{m+1,1}$ for
large enough and for small enough $m$. Also, we have $|k_{i}+m|=k_{i}+m$
for $m$ large enough. Let us take $m$ large enough and look at the
expression $$\frac{p_{m+1,1}p_{m+1,2}p_{m+1,3}}{p_{m,1}p_{m,2}p_{m,3}}.$$
By our assumption this fraction equals $1$ for every $m$. For $m$
large enough we have $p_{m+1,1}p_{m+1,2}=p_{m,1}p_{m,2}$, so $\frac{p_{m+1,3}}{p_{m,3}}=1$.
By the functional equation $\Gamma(z+1)=z\Gamma(z)$ we obtain
\[
1=\frac{p_{m+1,3}}{p_{m,3}}=\prod\limits _{i=1}^{n}\frac{\left(\frac{1}{2}+s_{i}+\frac{k_{i}+2m}{2}\right)}{\left(\frac{1}{2}-s_{i}+\frac{k_{i}+2m}{2}\right)}.
\]
Thus,
\[
\prod\limits _{i=1}^{n}\left(\frac{1}{2}+s_{i}+\frac{k_{i}+2m}{2}\right)=\prod\limits _{i=1}^{n}\left(\frac{1}{2}-s_{i}+\frac{k_{i}+2m}{2}\right)
\]
for large enough $m\in\ZZ$. Since both sides are polynomials in $m$,
the polynomials are equal. As a consequence, the zeros of these two
polynomials coincide, that is, for every $1\le i\le n$ there exists
$1\le j\le n$ such that $$\frac{1}{2}-s_{i}+\frac{k_{i}}{2}=\frac{1}{2}+s_{j}+\frac{k_{j}}{2}.$$
By our assumption $s_{i}$'s are purely imaginary and $k_{i}$'s are
integers. Thus, $-s_{i}=s_{j}$ and $k_{i}=k_{j}$. Note that $\bar{s_{i}}=-s_{i}$
and this means exactly that we can define $w(i)=j$ and $w(j)=i$
and $\chi_{j}=\bar{\chi_{i}}^{-1}=\chi_{w(i)}$. Therefore there exists
an involution $w\in S_{n}$ such that $w(\chi)=\bar{\chi}^{-1}$.\\
From the proof of the existence of an involution $w$ it follows also
that $\sum_{j=1}^{n}s_{j}=0$ and that the products $p_{m,2}=1$ and
$p_{m,3}=1$ for every $m\in\ZZ$. This establishes the existence
of an involution $w\in S_{n}$ such that $w(\chi)=\bar{\chi}^{-1}$.
It remains to establish the second property - existence of such an
involution such that in addition for every fixed point $w(j)=j$ of
the involution the corresponding integer $k_{j}$ is even. Note that
if $i$ is a fixed point of $w$ then $s_{i}=0$. Without loss of
generality assume that if $w(i)=i$ and $w(j)=j$, then $k_{i}\ne k_{j}$.
Otherwise we can define an involution $w'$ by $w'(i)=j,w'(j)=i$
and $w'(l)=w(l)$ for $l\ne i,j$ and the new involution $w'$ also
satisfies $w'(\chi)=\bar{\chi}^{-1}$.\\
Assume on the contrary that $w(i)=i$ but $k_{i}$ is odd. Then take
two consecutive products $p_{m,1}$ and $p_{m+1,1}$ for $m=\frac{-k_{i}-1}{2}$.
Observe that $\epsilon_{2m+k_{i}}=-\epsilon_{2(m+1)+k_{i}}$ and the
other terms appearing in the products $p_{m,1}$ and $p_{m+1,1}$
equal each other respectively. As a consequence, $p_{m+1,1}=-p_{m,1}$.
But from the written above we have $p_{m,2}=p_{m+1,2}=1$ and also
$p_{m,3}=p_{m+1,3}=1$ and thus $p_{m+1,1}=p_{m,1}=1$. Contradiction!\\
Therefore, if $w(i)=i$ then the integer $k_{i}$ is even, that is
$\chi_{i}(-1)=1$.
\end{proof}
A small modification of this proof gives a stronger theorem
\begin{thm}
\label{thm:converse B} Let $\chi=(\chi_{1},\chi_{2},...,\chi_{n})$
be a character of $B_{n}$ and suppose $\chi_{j}(z)=|z|_{\CC}^{s_{j}}(z/|z|)^{k_{j}}$
with $-\frac{1}{2}<Re(s_{j})<\frac{1}{2}$ and $k_{j}\in\ZZ$ for
every $1\le j\le n$. Suppose $(\pi,V)=I(\chi)$ is a smooth, generic,
irreducible representation of $G_{n}(\CC)$. Suppose
\[
\gamma\left(\frac{1}{2},\pi\times\chi',\psi\right)=1
\]
for every $\RR^{\times}$-distinguished unitary character $\chi':\CC^{\times}\to\RR^{\times}$.
Then there exists an involution $w\in S_{n}$ such that $w\chi=\overline{(\chi^{-1})}$.
Moreover, one can find an involution $w\in S_{n}$ such that $w\chi=\overline{(\chi^{-1})}$
and also for every fixed point $w(i)=i$ the integer $k_{i}$ is even.\end{thm}
\begin{proof}
By the same proof as in the previous theorem we obtain that for every
$1\le i\le n$ there exists $1\le j\le n$ such that $$\frac{1}{2}-s_{i}+\frac{k_{i}}{2}=\frac{1}{2}+s_{j}+\frac{k_{j}}{2}.$$
By subtracting $\frac{1}{2}$ from both sides of this equality and
taking real parts we can replace $s_{j}$ by $Re(s_{j})$. Thus we
can assume that for every $1\le i\le n$ we have $-\frac{1}{2}<s_{i}<\frac{1}{2}$
and also for every $1\le i\le n$ there exists $1\le j\le n$ such
that $-s_{i}+\frac{k_{i}}{2}=s_{j}+\frac{k_{j}}{2}$. Let us multiply
both sides of this equation by $2$ and replace $s_{i}$ by $2s_{i}$.
Therefore, we can assume that for every $1\le i\le n$ we have $-1<s_{i}<1$
and also for every $1\le i\le n$ there exists $1\le j\le n$ such
that $$-s_{i}+k_{i}=s_{j}+k_{j}.$$ Let us call this condition "antisymmetry
condition". The claim is that the "antisymmetry
condition" implies that there exists an involution $w\in S_{n}$
such that $w(\chi)=\bar{\chi}^{-1}$, that is, if $w(i)=j$ then $s_{i}=-s_{j}$
and $k_{i}=k_{j}$. The proof of the existence of an involution $w$
is by induction on $n$. Clearly, for $n=1$ the condition $-s_{1}+k_{1}=s_{1}+k_{1}$
gives us $s_{1}=0$ and thus the identity involution $w(1)=1$ works.
For a general $n$ it is enough that the "antisymmetry
condition" implies that there is a pair $i,j$ such
that $s_{i}=-s_{j}$ and $k_{i}=k_{j}$. Note that it can be that
$i=j$ and then $s_{i}=-s_{i}$ implies $s_{i}=0$.\\
Suppose on the contrary that there are $\{s_{i}\}_{i=1}^{n}\subset(-1,1)$
and $\{k_{i}\}_{i=1}^{n}\subset\ZZ$ that satisfy the "antisymmetry
condition", but there is no pair of indices $1\le i,j\le n$
that satisfy $s_{i}=-s_{j}$ and $k_{i}=k_{j}$. In particular, there
is some $1\le i\le n$ such that $-s_{1}+k_{1}=s_{i}+k_{i}$. By our
assumption $i>1$, so without loss of generality assume $i=2$. Let
us assume $s_{1}>0$. The proof in the case $s_{1}<0$ is similar
and $s_{1}=0$ is not possible by our assumption. We obtain $k_{1}-k_{2}=s_{1}+s_{2}$.
The left-hand side is an integer and $-1<s_{1}+s_{2}<2$. Thus $s_{1}+s_{2}=0$
or $s_{1}+s_{2}=1$. The case $s_{1}+s_{2}=0$ is not possible by
our assumption, thus $s_{1}+s_{2}=1$ and as a corollary $s_{2}>0$
and $k_{2}=k_{1}-1$. Similarly, there is some $1\le i\le n$ such
that $-s_{2}+k_{2}=s_{i}+k_{i}$. By the same argument we obtain $s_{i}>0$
and $k_{i}=k_{2}-1$. Thus $i\ne1,2$ and without loss of generality
we can assume $i=3$. Continuing in this manner we obtain an infinite
sequence of integers $k_{j}$ such that $k_{j}=k_{1}+(j-1)$. Contradiction!\\
Thus there is a pair of indices $1\le i,j\le n$ such that $s_{i}=-s_{j}$
and $k_{i}=k_{j}$. Removing them from our sequence of length $n$
we obtain a shorter sequence which satisfies the "antisymmetry
condition".\\
Thus, we have proved that there is an involution $w\in S_{n}$ such
that $w(\chi)=\bar{\chi}^{-1}$. The rest of the argument, that is,
the proof of the existence of an involution $w$ such that for every
fixed point $j$ of the involution the corresponding integer $k_{j}$
is even is the same as in the proof of the previous theorem.
\end{proof}
As a corollary, using the Tadic-Vogan classification of the unitary
dual of $G_{n}(\CC)$ we obtain the following
\begin{thm}
Let $\chi=(\chi_{1},\chi_{2},...,\chi_{n})$ be a character of $B_{n}$
and suppose $(\pi,V)=Ind(\chi)$ is smooth, generic, irreducible,
and unitary representation of $G_{n}(\CC)$. Suppose
\[
\gamma\left(\frac{1}{2},\pi\times\chi',\psi\right)=1
\]
 for every $\RR^{\times}$-distinguished unitary character $\chi':\CC^{\times}\to\RR^{\times}$.
Then there exists an involution $w\in S_{n}$ such that $w\chi=\overline{(\chi^{-1})}$.
Moreover, one can find an involution $w\in S_{n}$ such that $w\chi=\overline{(\chi^{-1})}$
and also for every fixed point $w(i)=i$ the integer $k_{i}$ is even.\end{thm}
\begin{proof}
Let us denote $\chi_{j}(z)=|z|_{\CC}^{s_{j}}(z/|z|)^{k_{j}}$ , where
$s_{j}\in\CC$ and $k_{j}\in\ZZ$. The theorem follows from Theorem
\ref{thm:converse B} and the fact that unitaricity of $Ind(\chi)$
implies $-\frac{1}{2}<Re(s_{j})<\frac{1}{2}$ for every $1\le j\le n$,
see \cite[Theorem A]{Tadic}.
\end{proof}
Finally, by \cite[Theorem 3.3.6]{Pan} we know
that an irreducible tempered representation $(\pi,V)$ of $G_{n}(\CC)$
is $G_{n}(\RR)$-distinguished if and only if there exists an involution
$w\in S_{n}$ such that $w\chi=\overline{(\chi^{-1})}$ and also for
every fixed point $w(i)=i$ the integer $k_{i}$ is even. Therefore,
an irreducible tempered representation $(\pi,V)$ of $G_{n}(\CC)$
is $G_{n}(\RR)$-distinguished if and only if
\[
\gamma\left(\frac{1}{2},\pi\times\chi',\psi\right)=1
\]
 for every $\RR^{\times}$-distinguished unitary character $\chi':\CC^{\times}\to\RR^{\times}$. 

\end{document}